\documentclass[12pt, a4paper]{amsart}
\usepackage[margin=1in]{geometry}  
\usepackage{amscd,latexsym,amsthm,amsfonts,amssymb,amsmath,amsxtra}
\usepackage{mathrsfs}
\usepackage{hyperref}
\usepackage{pdfsync}
\usepackage[all,cmtip]{xy}
\usepackage{graphicx}
\usepackage{float}
\usepackage{color}
\usepackage{hyperref}[colorlinks,linkcolor=blue]
\usepackage{tikz}
\usepackage{tikz-cd}
\usetikzlibrary{decorations.pathreplacing}

\usepackage{ytableau}
\usepackage{abstract}
\usepackage{appendix}
\usepackage{booktabs}
\usepackage{upgreek}
\numberwithin{equation}{section}
\newtheorem{thm}{Theorem}[section]
\newtheorem{cor}[thm]{Corollary}
\newtheorem{prop}[thm]{Proposition}
\newtheorem{lem}[thm]{Lemma}

\newtheorem{defn}[thm]{Definition}

\newtheorem{examp}[thm]{Example}

\newcommand{\blam}{{\boldsymbol{\lambda}}}

\newcommand{\br}{{\mathbf{r}}}
\newcommand{\bc}{{\mathbf{c}}}

\newcommand{\BC}{{\mathbb {C}}}

\newcommand{\BH}{{\mathbb {H}}}

\newcommand{\BN}{{\mathbb {N}}}

\newcommand{\BR}{{\mathbb {R}}}
\newcommand{\BS}{{\mathbb {S}}}

\newcommand{\BZ}{{\mathbb {Z}}}

\newcommand{\CK}{{\mathcal {K}}}

\newcommand{\CO}{{\mathcal {O}}}
\newcommand{\CP}{{\mathcal {P}}}
\newcommand{\CQ}{{\mathcal {Q}}}
\newcommand{\CR}{{\mathcal {R}}}

\newcommand{\CU}{{\mathcal {U}}}

\newcommand{\CZ}{{\mathcal {Z}}}

\newcommand{\RA}{{\mathrm {A}}}

\newcommand{\RS}{{\mathrm {S}}}

\newcommand{\fb}{\mathfrak{b}}

\newcommand{\fg}{\mathfrak{g}}
\newcommand{\fh}{\mathfrak{h}}

\newcommand{\ft}{\mathfrak{t}}

\newcommand{\fgl}{{\mathfrak{gl}}}

\newcommand{\GL}{{\mathrm{GL}}}
\newcommand{\G}{{\mathrm{G}}}
\newcommand{\U}{{\mathrm{U}}}

\newcommand{\A}{{\mathrm{A}}}

\newcommand{\Ind}{{\mathrm{Ind}}}
\newcommand{\Irr}{{\mathrm{Irr}}}

\newcommand{\Inn}{{\mathrm{Inn}}}

\newcommand{\Ad}{{\mathrm{Ad}}}

\newcommand{\Lie}{{\mathrm{Lie}}}
\newcommand{\Ann}{{\mathrm{Ann}}}

\newcommand{\sgn}{{\mathrm{sgn}}}
\newcommand{\Rep}{{\mathrm{Rep}}}
\newcommand{\Nil}{{\mathrm{Nil}}}

\newcommand{\set}[2]{\{#1\,|\,#2\}}

\newcommand{\defmap}[5]{
           \begin{equation}
              \begin{aligned}
                   #1:\quad  & #2 &\longrightarrow &\quad #3 \\
                      \quad  & #4    &\longmapsto  &\quad #5
              \end{aligned}
           \end{equation}
          }

\renewcommand{\bar}{\overline}
\renewcommand{\tilde}{\widetilde}

\author[Q.~Wang]{Qiutong Wang}

\address{SCHOOL OF MATHEMATICAL SCIENCES, ZHEJIANG UNIVERSITY, HANGZHOU 310058, PEOPLE'S REPUBLIC OF CHINA}\email{wang.qt@zju.edu.cn}

\begin{document}
\ytableausetup{centertableaux}
\title[COUNTING IRREDUCIBLE REPRESENTATIONS]{COUNTING IRREDUCIBLE REPRESENTATIONS OF GENERAL LINEAR GROUPS AND UNITARY GROUPS}
\maketitle

\begin{abstract}
    Let $G$ be a general linear group over $\BR$, $\BC$, or $\BH$, or a real unitary group. In this paper, we precisely describe the number of isomorphism classes of irreducible Casselman--Wallach representations of $G$ with a given infinitesimal character and a given complex associated variety, expressed in terms of certain combinatorial data called painted Young diagrams and assigned Young diagrams.
\end{abstract}

\tableofcontents

\section{Introduction}

\subsection{Background and goals}

Let $G$ be a real reductive Lie group. Due to the work of Harish-Chandra, the set $\Irr(G)$ of infinitesimal equivalence classes of irreducible admissible representations (equivalently, isomorphism classes of Casselman--Wallach representations \cite[Chapter 11]{Wal}) of $G$ admits a partition according to infinitesimal characters:
$$\Irr(G) = \bigsqcup_{\lambda} \Irr_{\lambda}(G),$$
where each subset $\Irr_{\lambda}(G)$ is finite (\cite[Theorem 7]{Har}), the number of irreducible representations in $\Irr_{\lambda}(G)$ can be effectively computed by listing the corresponding Langlands parameters (see, for example, \cite[Section~6]{ALTV}). 

Another invariant of irreducible Casselman--Wallach representations is the complex associated variety (to be introduced in Section \ref{2.1}), which turns out to be the closure of a unique nilpotent orbit. This notion plays an important role in several areas, including Kirillov's orbit philosophy and the Kazhdan--Lusztig theory. 

In this paper, we study the further partition of each $\Irr_{\lambda}(G)$ according to the complex associated varieties:
$$\Irr_{\lambda}(G) = \bigsqcup_{\CO}\Irr_{\lambda}(G;\CO),$$
and we give a precise description of the size of each $\Irr_{\lambda}(G;\CO)$ when $G$ is a general linear group over $\BR$, $\BC$, or $\BH$, or a real unitary group. 

For classical groups of types $B$, $C$, and $D$, the same counting method also applies; however, the double cells (introduced in Section \ref{2.3} ) of the corresponding Weyl group representations are much more complicated to describe, so we are not able to obtain an equally clean formula for the counting results in those cases.

The counting method employed in this paper was developed in \cite{BMSZ} by Barbasch, Ma, Sun, and Zhu, based on the theory of coherent continuation representations. Coherent families of representations were first introduced by Schmid \cite{Sch}, a standard reference is \cite[Chapter~7]{Vog81}. For completeness, we recall the basic definitions in Section~\ref{2}.

\subsection{Representation-theoretic and combinatorial notations}\label{1.2}
Let $G$ be a real reductive Lie group in Harish-Chandra's class, which is defined by the following properties:
\begin{itemize}
    \item the corresponding Lie algebra $\Lie (G)$ is reductive;
    \item $G$ has finitely many connected components;
    \item the connected subgroup $G_{ss}$ of $G$ corresponding to $\Lie(G)_{ss} = [\Lie(G),\Lie(G)]$ has a finite center;
    \item for any $g \in G$, the adjoint action $\Ad(g): \fg \to \fg$, lies in the group of inner automorphisms $\Inn(\fg)$. 
 \end{itemize}

We say that $G$ is linear if it admits a faithful finite-dimensional representation. Let $\Rep(G)$ denote the category of Casselman--Wallach representations of $G$, whose Grothendieck group (with coefficients in $\BC$) is denoted by $\CK(G)$. Let $\Irr (G)$ be the set of isomorphism classes of irreducible Casselman--Wallach representations of $G$, which form a basis of $\CK(G)$.

The universal enveloping algebra of $\fg$ is denoted by $\CU(\fg)$, and its center is denoted by $\CZ(\fg)$. Let $^{a}\fh$ denote the abstract Cartan subalgebra of $\fg$ (recall that for every Borel subalgebra $\fb$ of $\fg$, there is an identification $^{a}\fh = \fb/[\fb,\fb]$). Write 
$$\Delta^+ \subseteq \Delta \subseteq {^{a}\fh^*} \ \ \textrm{and} \ \ \check{\Delta}^+ \subseteq \check{\Delta} \subseteq {^{a}\fh}$$
for the positive root system, the root system, the positive coroot system, and the coroot system, respectively, for the reductive Lie algebra $\fg$. Here we use a superscript $*$ to denote the linear dual of a vector space. Let $Q_{\fg} \subseteq {^a\fh}^*$ and $Q^{\fg} \subseteq {^a\fh}^*$ denote the root lattice and the weight group of $\fg$, respectively. Let $W \subseteq \GL(^{a}\fh)$ denote the abstract Weyl group. By the Harish-Chandra isomorphism, there is a 1-1 correspondence between $W$-orbits of $\nu \in {^{a}\fh^*}$ and algebraic characters $\chi_{\nu}:\CZ(\fg) \to \BC$. We say that an ideal of $\CU(\fg)$ has infinitesimal character $\nu$ if it contains the kernel of $\chi_{\nu}$.

Let $\Nil(\fg)$ (resp. $\Nil(\fg^*)$) denote the set of nilpotent elements in $[\fg,\fg]$ (resp. $[\fg,\fg]^*$), and define
\begin{equation}
   \bar{\Nil}(\fg) := \Inn(\fg) \backslash \Nil(\fg), \ \  \bar{\Nil}(\fg^*) := \Inn(\fg) \backslash \Nil(\fg^*),
\end{equation}
the set of nilpotent orbits in $\Nil(\fg)$ and $\Nil(\fg^*)$. The Killing form on $[\fg,\fg]$ yields an identification $\bar{\Nil}(\fg) = \bar{\Nil}(\fg^*)$.

For each Casselman--Wallach representation $V \in \Rep(G)$, we define in Section \ref{2.1} an $\Inn(\fg)$-stable Zariski closed subvariety $\mathrm{AV}_{\mathbb{C}}(V)$ of $\Nil(\fg^*)$, which is called the complex associated variety of $V$. Moreover, if $V$ is irreducible, we explain that $\mathrm{AV}_{\mathbb{C}}(V)$ is the closure of a unique nilpotent orbit.

Suppose $S$ is a $\Inn(\fg)$-stable Zariski closed subset of $\Nil(\fg^*)$, denote by $\Rep_S(G)$ the category of Casselman--Wallach representations of $G$ whose complex associated variety is contained in $S$.

For every $\nu \in {^{a}\fh^*}$, let $\Rep_{\nu}(G)$ and $\Rep_{\nu, S}(G)$ denote the full subcategories of $\Rep(G)$ and $\Rep_S(G)$, respectively, whose objects consisting of Casselman--Wallach representations with generalized infinitesimal character $\nu$. Denote by $\CK_{\nu}(G)$, and $\CK_{\nu,S}(G)$ the Grothendieck groups of $\Rep_{\nu}(G)$, and $\Rep_{\nu,S}(G)$, respectively. The set of isomorphism classes of irreducible objects in $\Rep_{\nu}(G)$ and $\Rep_{\nu,S}(G)$ will be denoted by $\Irr_{\nu}(G)$ and $\Irr_{\nu,S}(G)$, respectively. 

Suppose $\nu \in {^{a}\fh^*}$ and $\CO$ is a nilpotent orbit in $\fg^*$,   denote by $\Irr_{\nu}(G;\CO)$ the subset of $\Irr(G)$ consisting of isomorphism classes of irreducible representations with infinitesimal character $\nu \in {^{a}\fh^*}$ and complex associated variety $\bar{\CO}$ (the Zariski closure of $\CO$). Then we observe that for any $\Inn(\fg)$-stable Zariski closed subset $S$ of $\Nil(\fg^*)$
\begin{equation}\label{(1.3)}
    \Irr_{\nu,S}(G) = \bigsqcup_{\substack{\CO \in \bar{\Nil}(\fg) \\ \CO \subseteq S}} \Irr_{\nu}(G;\CO).
\end{equation}

This article aims to describe the number of elements in $\Irr_{\nu}(G;\CO)$ when $G$ is a general linear group over $\BR$, $\BH$, or $\BC$, or a real unitary group. We attach a label $\star$ to each class of these groups, as in the following table ($n,p,q \in \BN $).

\begin{center}
   \begin{tabular}{c|c|c}
      \toprule
      Label $\star $ & Classical Lie Group G & Complex Lie Group $G_{\BC}$    \\
      \midrule
      $A^{\BR}$      & $\GL_n(\BR)$          & $\GL_n(\BC)$                   \\
      $A^{\BH}$      & $\GL_{\frac{n}{2}}(\BH)$ \  ($n$ is even)      & $\GL_n(\BC)$                   \\
      $A^{\BC}$      & $\GL_n(\BC)$          & $\GL_n(\BC) \times \GL_n(\BC)$ \\
      $A$            & $\U(p,q)$              & $\GL_n(\BC)$ \ $(n = p + q)$                   \\
      \bottomrule
   \end{tabular}
\end{center}

However, to deal with the real unitary case, we need to introduce an additional setting where $\star = \tilde{A}$, $G = \tilde{\U}(p,q)$ and $G_{\BC} = \GL_{n}(\BC)$ ($n = p+q$). Here, $\tilde{\U}(p,q)$ denotes the (linear) double cover of $\U(p,q)$, defined via the pullback diagram.

\begin{center}
    \begin{tikzcd}
        {\tilde{\mathrm{U}}(p,q)} \arrow[rr] \arrow[dd, "\mathrm{det}^{\frac{1}{2}}"] &  & {\mathrm{U}(p,q)} \arrow[dd, "\mathrm{det}"] \\
        & &  \\
        \mathbb{S}^{1} \arrow[rr, "\left( - \right)^{2}"]  &  & \mathbb{S}^{1}                         
    \end{tikzcd} 
\end{center}

Note that in each case $\star = A^{\BR}, A^{\BH}, A^{\BC}, A, \tilde{A}$, there exist a natural map $\iota: G \to G_{\BC}$ which satisfies all the assumptions that will be introduced in Section \ref{2.2}.

For a Young diagram $\iota$, write
$$\br_1(\iota) \geq \br_2(\iota) \geq \br_3(\iota) \geq \cdots$$
for its row lengths, and similarly
$$\bc_1(\iota) \geq \bc_2(\iota) \geq \bc_3(\iota) \geq \cdots$$
for its column lengths. Let $|\iota|:= \sum_{i=1}^{\infty}\br_i(\iota) \in \BN$ denote the total size of $\iota$, and let $\mathrm{YD}_{n}$ ($n \geq 0$) denote the set of all Young diagrams with total size $n$. For a non-negative integer $n$, we define the set of partitions of $n$ by the descending sequences of positive integers $[d_1,\cdots,d_k]$ ($k \geq 1$), such that $\sum_{i=1}^{k}d_i = n$.  We identify this set with the set of partitions of $n$ by associating each Young diagram with the partition given by row lengths. By abuse of notation, we also use Young diagrams to represent partitions.

For any Young diagram $\iota$, we introduce the set $\mathrm{Box}(\iota)$ of boxes of $\iota$ as follows:
\begin{equation}
   \mathrm{Box}(\iota) := \set{(i,j) \in \BN^{+} \times \BN^{+}}{j \leq \br_i(\iota)}.
\end{equation}
A subset of $\BN^{+} \times \BN^{+}$ of this form also constitutes the Young diagram $\iota$.

We introduce five symbols $\bullet, s, r, c \ \textrm{and} \ d$, and make the following definition.

\begin{defn}(cf.~\cite[Definition~2.9]{BMSZ})
A painting on a Young diagram $\iota$ is a map (i.e., we place a symbol in each box)
\[
\CP : \mathrm{Box}(\iota) \to \{\bullet, s, r, c, d\}
\]
satisfying the following properties:
\begin{enumerate}
\item If we remove the boxes painted with $d$, $c$ or $d$, $r$, $c$ or $d$, or $s$, $r$, $c$ or $d$, the remaining boxes still form a Young diagram;

\item Every row of $\iota$ contains at most one box painted with $s$, and at most one box painted with $r$;

\item Every column of $\iota$ contains at most one box painted with $c$, and at most one box painted with $d$.
\end{enumerate}

A \emph{painted Young diagram} is a pair $(\iota, \CP)$ consisting of a Young diagram $\iota$ and a painting $\CP$ on $\iota$.
\end{defn}

\begin{defn}
   Suppose that $\star \in \{A^{\BR}, A^{\BH}\}$. A painting $\CP$ on a Young diagram $\iota$ has type $\star$ if
   \begin{enumerate}
      \item the symbols of $\CP$ are in
            $$\left\{
               \begin{array}{lr}
                  \{\bullet,c,d\}, & \textrm{if $\star = A^{\BR}$},\\                   
                  \{\bullet\},     & \textrm{if $\star = A^{\BH}$};
               \end{array}
               \right.
            $$
      \item every column of $\iota$ has an even number of boxes painted with $\bullet$.
   \end{enumerate}
    \end{defn}

    \begin{defn}
           A painting $\CP$ on a Young diagram $\iota$ has type $A$ if
   \begin{enumerate}
      \item the symbols of $\CP$ are in $\{\bullet, s ,r\}$;
      \item every row of $\iota$ has an even number of boxes painted with $\bullet$.
   \end{enumerate}
    \end{defn}

   Denote by $\mathrm{P}_{\star}(\iota)$ the set of paintings of type $\star$ on $\iota$, for $\star \in \{A^{\mathbb{R}}, A^{\mathbb{H}}, A\}$.
   Now suppose that $\iota$ is a Young diagram and $\CP$ is a painting on $\iota$ of type $A$. Define the signature of $\CP$ to be the pair of non-negative integers
   \begin{equation}
    \left(p_{\CP}, q_{\CP}\right) := \left(\frac{\sharp(\CP^{-1}(\bullet))}{2} + \sharp(\CP^{-1}(s)), \frac{\sharp(\CP^{-1}(\bullet))}{2}+ \sharp(\CP^{-1}(r))\right),
   \end{equation}
   for every $p,q \in \BN$ such that $p + q = |\iota|$, we define
   \begin{equation}
        \mathrm{P}_{A}^{p,q}(\iota) := \set{\CP \in \mathrm{P}_{A}(\iota)}{(p_{\CP},q_{\CP}) = (p,q)}.       
   \end{equation}

We also introduce the notion of assigned Young diagrams.

\begin{defn}
   For a Young diagram $\iota$, and a partition $[d_1, \cdots, d_N]$ of $|\iota|$.  An assignment of type $[d_1,d_2, \cdots, d_N]$ on $\iota$ is a map (we place a positive integer in each box)
   $$\CQ: \mathrm{Box}(\iota) \to \{1,2, \cdots,N\} $$
   with the following properties

   \begin{enumerate}
      \item for each $i \in \{1,2,\cdots,N\}$, the preimage $\CP^{-1}(i)$ has exactly $d_i$ elements;
      \item for each $1 \leq n \leq N$, if we remove the boxes assigned with $\{n+1, \cdots, N\}$, the reminder still constitutes a Young diagram;
      \item each positive integer occurs at most once in each column.
   \end{enumerate}
   An \emph{assigned Young diagram} of type $[d_1,d_2, \cdots, d_N]$ is a pair $(\iota,\CQ)$ consisting of a Young diagram $\iota$ and an assignment $\CQ$ of type $[d_1,d_2, \cdots, d_N]$ on $\iota$. Denote by $\RA_{[d_1,d_2,\cdots,d_N]}(\iota)$ the set of all assignments on $\iota$ of type $[d_1,d_2,\cdots,d_N]$.
\end{defn}

\begin{examp}
    The following represents an assigned Young diagram of type $[4,3,2,2,1]$.
    \[
    \begin{ytableau}
        1&1&1&1&2\\
        2&2&3\\
        3&4&5\\
        4
    \end{ytableau}
    \]
    Each of the following does not represent an assigned Young diagram.
    \[
    \begin{ytableau}
        1\\
        1
    \end{ytableau}
    \qquad   
    \begin{ytableau}
        1&3\\
        2&2
    \end{ytableau}
    \]
\end{examp}

\subsection{The main results}
In the statement of the following theorems, we will use the usual identifications $^{a}\fh^* = \BC^n \ \textrm{or} \ \BC^n \times \BC^n$ introduced in Section \ref{3.1}, and identify nilpotent orbits in Lie algebras of type $A$ with partitions or Young diagrams in the usual way. For a partition $\iota$, denote by $\CO_\iota$ the corresponding nilpotent orbit, and for a nilpotent orbit $\CO$, denote by $\iota(\CO)$ the corresponding partition. We also define the ``row by row” union of Young diagrams $\mathop{\sqcup}\limits^r$ by 
$$[d_1,\cdots,d_k] \mathop{\sqcup}\limits^r [d_1' ,\cdots,d_{k'}'] =  [d_1 +d_1', \cdots,d_{k} + d_{k}', 0+d_{k+1}', \cdots, 0+d_{k'}']$$
where $k' \geq k$ are positive integers. For example, $[5,3,1] \mathop{\sqcup}\limits^r [4,3,3,3] = [9,6,4,3]$.

\begin{thm}\label{R}
   Suppose $G = \GL_n(\BR)$ ($n \in \BN$), and let $\CO \in \bar{\Nil}(\fg^*)$.
   \begin{enumerate}
        \item If $\nu \in {^{a}\fh^*} = \BC^n$ is integral, that is, the differences of its coordinates are integral, then its coordinates can be permuted such that 
        \[
        \nu = (\underbrace{\lambda_1, \cdots, \lambda_1}_{d_1}, \underbrace{\lambda_2, \cdots, \lambda_2}_{d_2}, \cdots, \underbrace{\lambda_k, \cdots, \lambda_k}_{d_k}) \in \BC^n,
        \]
        where $[d_1, d_2, \cdots , d_k]$ is a partition of $n$, and the $\lambda_i \in \BC$ satisfy the condition $\lambda_i - \lambda_j \in \BZ \setminus \{0\}$ for any $i \neq j$. Then
        \begin{equation}
            \sharp(\Irr_\nu(G;\CO)) = \sharp\left(\mathrm{P}_{A^{\BR}}(\iota(\CO))\right)\cdot \sharp\left(\RA_{[d_1,\cdots,d_k]}(\iota(\CO))\right).
        \end{equation}
        \item For an arbitrary $\nu \in {^{a}\fh}^*$, its coordinates can be permuted such that 
        \[
        \nu = (\blam_1, \cdots, \blam_r) \in \BC^n,
        \]
        \[
        \blam_i = (\underbrace{\lambda_{i,1}, \cdots, \lambda_{i,1}}_{d_{i,1}}, \cdots, \underbrace{\lambda_{i,k_i}, \cdots, \lambda_{i,k_i}}_{d_{i,k_i}}) \in \BC^{e_i} \ \ (e_i \geq 1),
        \]  
        where each $\blam_i$ is integral but $(\blam_i,\blam_j) \in \BC^{e_i+e_j}$ is not integral for any $i \neq j$, $[d_{i,1}, \cdots d_{i,k_i} ]$ is a partition of $e_i$, and the condition $\lambda_{i,a} - \lambda_{i,b} \in \BZ \setminus \{0\}$ holds for any $a,b \in \{1,\cdots,k_i\}$ with $a \neq b$.  Then
        \begin{equation}
            \sharp(\Irr_{\nu}(G;\CO)) = \sum_{\substack{(\iota_1,\cdots,\iota_r) \in \mathrm{YD}_{e_1} \times \cdots \times \mathrm{YD}_{e_r} \\ \iota_1 \mathop{\sqcup}\limits^r \iota_2 \cdots \mathop{\sqcup}\limits^r  \iota_r = \iota(\CO) }}\prod_{i=1}^{r}\sharp(\Irr_{\blam_i}(\GL_{e_i}(\BR);\CO_{\iota_i})).
        \end{equation}
    \end{enumerate}
\end{thm}

\begin{examp}
    \begin{enumerate}
        \item Recall that a partial order is defined on the set of nilpotent orbits by $\CO \leq \CO'$ if $\bar{\CO} \subseteq \bar{\CO'}$. It can also be described in terms of corresponding partitions as follows: 
        \[
            [d_1, d_2, \cdots , d_k] \leq [d_1', d_2', \cdots , d_{k'}']
        \]
        if the condition
        \[
            \sum_{1 \leq j \leq m} d_{j} \leq \sum_{1 \leq j \leq m} d_{j}' \  \ \textrm{for $1 \leq m \leq \mathrm{max}\{k,k'\}$},
        \]
        is satisfied. Here, zeros are added to the shorter partition if necessary.
        With the notation as in the integral case of Theorem \ref{R}, it is easy to verify that $\CO([d_1, d_2, \cdots, d_k])$ is the unique minimal nilpotent orbit in this partial order, whose closure is a complex associated variety of an irreducible Casselman--Wallach representation of $\GL_n(\BR)$ with infinitesimal character $\nu$.
        \item Suppose $G = \GL_n(\BR)$, recall that an irreducible Casselman--Wallach representation is called generic if it admits a Whittaker model, or equivalently, its complex associated variety is the closure of the principal orbit (\cite{Vog78}), whose Young diagram is:
        \[
                \begin{ytableau}
                ~ & & & \cdots & ~ 
                \end{ytableau}.
        \]
        
        When $\nu = (\lambda_1,\cdots,\lambda_n)$ is regular and integral, which means $\lambda_i - \lambda_j \in \mathbb{Z} \backslash \{0\}$ for all $i \neq j$, there is only one assignment of type $[1,1,\cdots,1]$, and exactly $n+1$ paintings of type $A_{\BR}$ on this Young diagram. Consequently, the number of generic representations of a regular integral infinitesimal character is $n+1$.
    \end{enumerate}
\end{examp}

\begin{thm}\label{H}
    Let $G= \GL_{\frac{n}{2}}(\BH)$ ($n \in \BN$ is even), and let $\CO \in \bar{\Nil}(\fg^*)$.
    \begin{enumerate}
        \item If $\nu \in {^a\fh^*} = \BC^n$ is integral, that is, the differences of its coordinates are integral, then its coordinates can be permuted such that 
        \[
        \nu =  (\underbrace{\lambda_1, \cdots, \lambda_1}_{d_1}, \underbrace{\lambda_2, \cdots, \lambda_2}_{d_2}, \cdots, \underbrace{\lambda_k, \cdots, \lambda_k}_{d_k} ) \in \BC^n,
        \]
        where $[d_1, d_2, \cdots, d_k]$ is a partition of $n$, and the $\lambda_i \in \BC$ satisfy the condition $\lambda_i - \lambda_j \in \BZ \backslash \{0\}$ for any $i \neq j$. Then
        \begin{equation}
            \sharp(\Irr_{\nu}(G;\CO)) = \sharp\left(\mathrm{P}_{A^{\BH}}(\iota(\CO))\right)\cdot \sharp\left(\A_{[d_1,\cdots,d_k]}(\iota(\CO))\right).
        \end{equation}
        \item For an arbitrary $\nu \in {^{a}\fh}$, its coordinates can be permuted such that
        \[    
        \nu = (\blam_1, \cdots, \blam_r) \in \BC^n,
        \]
        \[
            \blam_i = (\underbrace{\lambda_{i,1}, \cdots, \lambda_{i,1}}_{d_{i,1}},\cdots,\underbrace{\lambda_{i,k_i},\cdots,\lambda_{i,k_i}}_{d_{i,k_i}}) \in \BC^{e_i} \ \ (e_i \geq 1),
         \] 
        where each $\blam_i$ is integral but $(\blam_i,\blam_j) \in \BC^{e_i+e_j}$ is not integral for any $i \neq j$, $[d_{i,1}, \cdots, d_{i,k_i}]$ is a partition of $e_i$, and the condition $\lambda_{i,a} - \lambda_{i,b} \in \BZ \backslash \{0\}$ holds for any $a,b \in \{ 1,\cdots,k_i \}$ with $a \neq b$. Then
        \begin{equation}
            \sharp(\Irr_{\nu}(G;\CO)) = \left\{
            \begin{aligned}
                &\sum_{\substack{(\iota_1,\cdots,\iota_r) \in \mathrm{YD}_{e_1} \times \cdots \times \mathrm{YD}_{e_r} \\ \iota_1 \mathop{\sqcup}\limits^r \cdots  \mathop{\sqcup}\limits^r \iota_r = \iota(\CO)}} \prod_{i=1}^r \sharp\left(\Irr_{\blam_i}\left(\GL_{\frac{e_i}{2}}(\BH);\CO_{\iota_i}\right)\right), & \textrm{if each $e_i$ is even;}\\
                &0 , & \textrm{otherwise}.
            \end{aligned}
            \right.
        \end{equation}
    \end{enumerate}
\end{thm}

\begin{examp}
    \begin{enumerate}
    \item  Suppose $G = \GL_{4}(\BH)$, $\nu = (1,1,1,2,2,2,3,3)$.
           Then the complex associated variety of irreducible Casselman--Wallach representations with this infinitesimal character can only be the closure of the nilpotent orbit $\CO$, whose Young diagram is
           \[
                \begin{ytableau}
                ~ & & & \\
                 & & & 
                \end{ytableau}.
           \]
            There is a unique assignment of type $[3,3,2]$ on this diagram, given by
            \[
                \begin{ytableau}
                    1 & 1 & 1 & 2\\
                    2 & 2 & 3 & 3
                \end{ytableau}.
            \]
            This implies that $\sharp(\Irr_{\nu}(G;\CO)) = 1$, and $\sharp(\Irr_{\nu}(G)) = 1$.

    \item   Suppose $G = \GL_{5}(\BH)$, $\nu = (1,1,1,1,2,2,2,3,3,4)$.       
            There are two nilpotent orbits $\CO_1$ and $\CO_2$, whose closures can occur as the complex associated variety of an irreducible Casselman--Wallach representation with this infinitesimal character. The corresponding Young diagrams are, respectively
            
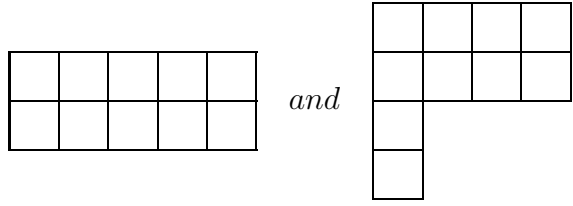
\[
                \begin{ytableau}
                    ~& & & & \\
                    & & & &
                \end{ytableau}
                \quad and \quad 
                \begin{ytableau}
                    ~&~&~&~\\
                    ~&~&~&~\\
                    ~\\
                    ~
                \end{ytableau}.
            \]
            
            On the first diagram, there are two assignments of type $[4,3,2,1]$, given by
            \[
                \begin{ytableau}
                    1&1&1&1&2\\
                    2&2&3&3&4
                \end{ytableau}
                \quad and \quad 
                \begin{ytableau}
                    1&1&1&1&3\\
                    2&2&2&3&4
                \end{ytableau}.
            \]
            
            On the second diagram, there is only one assignment of type $[4,3,2,1]$, given by
            \[
                \begin{ytableau}
                    1&1&1&1\\
                    2&2&2&3\\
                    3\\
                    4
                \end{ytableau}.
            \]
            This implies that $\sharp(\Irr_{\nu}(G;\CO_1)) = 2$ and $\sharp(\Irr_{\nu}(G;\CO_2)) = 1$.
    \end{enumerate}
\end{examp}

\begin{thm}\label{C}
    Suppose $G = \GL_{n}(\BC)$ ($n \in \BN$), and let $\CO = \CO_{\iota} \times \CO_{\iota'} \in \bar{\Nil}(\fg^*)$.
    \begin{enumerate}
        \item If $\nu \in {^a\fh^*} = \BC^n \times \BC^n$ is integral, that is, in each factor $\BC^n$, the differences of its coordinates are integral, then its coordinates can be permuted via the Weyl group $\mathrm{S}_n \times \mathrm{S}_n$ such that
        \[
        \nu =  (\underbrace{\lambda_1, \cdots, \lambda_1}_{d_1}, \cdots, \underbrace{\lambda_k, \cdots, \lambda_k}_{d_k}, \underbrace{\lambda_1', \cdots, \lambda_1'}_{d_1'}, \cdots, \underbrace{\lambda_l', \cdots, \lambda_l'}_{d_{l}'} ) \in \BC^n \times \BC^n, 
        \]
        where $[d_1,d_2,\cdots,d_k]$ and $[d_1',d_2',\cdots,d_l']$ are partitions of $n$, and the $\lambda_i \in \BC$ satisfy the condition $\lambda_i - \lambda_j, \ \lambda_i' - \lambda_j' \in \BZ \backslash \{0\}$ for any $i \neq j$. Then
        \begin{equation}
            \sharp(\Irr_{\nu}(G;\CO)) = \left\{ 
            \begin{aligned}
                &\sharp\left(\A_{[d_1,\cdots,d_k]}(\iota)\right) \cdot \sharp\left(\A_{[d_1',\cdots,d_l']}(\iota')\right) \cdot \delta_{\iota,\iota'} , & \textrm{if $\lambda_1 - \lambda_1' \in \BZ$};\\
                &0, & \textrm{if $\lambda_1 - \lambda_1' \notin \BZ$}
            \end{aligned}
            \right.
        \end{equation}
        where $$\delta_{\iota,\iota'} = \left\{
        \begin{aligned}
            &0, & \textrm{if $\iota \neq  \iota'$;}\\
            &1, & \textrm{if $\iota = \iota'$}.
        \end{aligned}
        \right.$$
        \item For an arbitrary $\nu \in {^{a}\fh^*}$, its coordinates can be permuted via the Weyl group such that
        \[
        \nu = (\blam_1, \cdots, \blam_r, \blam_1', \cdots, \blam_s') \in \BC^n \times \BC^n,
        \]
        \begin{align}
            &\blam_i = (\underbrace{\lambda_{i,1}, \cdots, \lambda_{i,1}}_{d_{i,1}}, \cdots, \underbrace{\lambda_{i,k_i}, \cdots, \lambda_{i,k_i}}_{d_{i,k_i}}) \in \BC^{e_i},\\
            &\blam_j' = (\underbrace{\lambda'_{j,1}, \cdots, \lambda'_{j,1}}_{d_{j,1}'}, \cdots, \underbrace{\lambda'_{j,l_j}, \cdots, \lambda'_{j,l_j}}_{d_{j,l_j}'}) \in \BC^{e_j'},
        \end{align}
        where each $\blam_i$ (resp. $\blam'_j$) is integral but $(\blam_i,\blam_j) \in \BC^{e_i+e_j}$ (resp. $(\blam'_i,\blam'_j) \in \BC^{e'_i+e'_j}$) is not integral for any $i \neq j$. $[d_{i,1}, \cdots, d_{i,k_i}]$ (resp. $[d_{j,1}', \cdots, d_{j,l_j}']$) forms a partition of $e_i$ (resp. $e'_j$). The condition $\lambda_{i,a} - \lambda_{i,b} \in \BZ \setminus \{0\}$ (resp. $\lambda_{j,a}' - \lambda_{j,b}' \in \BZ \setminus \{0\}$) holds for any $a,b \in \{1,\cdots.k_i\}$ (resp. $a,b \in \{1,\cdots,l_j \}$) with $a \neq b$. 
    
        Then there exist irreducible Casselman--Wallach representations with this infinitesimal character only if $r = s$. We can permute the coordinates of $\nu$ under the Weyl group such that $e_i = e_i'$ and $\lambda_{i,1} - \lambda_{i,1}' \in \BZ$ for each $i \in \{1, 2, \cdots, r\}$. In this case, 
        \begin{equation}
            \sharp(\Irr_{\nu}(G;\CO)) = \sum_{\substack{(\iota_1,\cdots,\iota_r), (\iota_1',\cdots,\iota_r') \in \mathrm{YD}_{e_1} \times \cdots \times \mathrm{YD}_{e_r}\\ \mathop{\sqcup}\limits^r \iota_i = \iota, \ \mathop{\sqcup}\limits^r \iota_i' = \iota' }} \prod_{i=1}^{r} \sharp(\Irr_{\blam_i}(\GL_{e_i}(\BC);\CO_{\iota_i} \times \CO_{\iota_i'})).
        \end{equation}
    \end{enumerate}
\end{thm}

\begin{examp}
    Suppose $G = \GL_n(\BC)$, $\nu \in {^{a}\fh^*}$ is regular integral, and $\CO = \CO_{\min} \times \CO_{\min} \in \bar{\Nil}(\fg)$, where $\fg$ is identified with $\Lie(G) \times \Lie(G)$ as in section \ref{3.1}, the corresponding Young diagram of $\CO_{\min}$ is
    
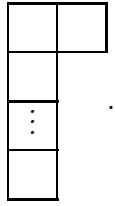
\[
    \begin{ytableau}
        ~&~\\
        ~\\
        \vdots\\
        ~
    \end{ytableau}.
    \]
    There are $n-1$ assignments of type $[\underbrace{1,1,\cdots,1}_n]$ on it, given by
    
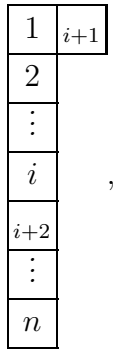
\[
    \begin{ytableau}
        1&_{i+1}\\
        2\\
        \vdots\\
        i\\
        _{i+2}\\
        \vdots\\
        n
    \end{ytableau},
    \]
    for $i = 1,2,\cdots,n-1$.
    This implies $\sharp(\Irr_{\nu}(G;\CO_{\min} \times \CO_{\min})) = (n-1)^2$, which is consistent with the results in \cite{BCLS}.
\end{examp}

\begin{thm}\label{U}
    Suppose $G = \U(p,q)$ ($p, q \in \BN$), and let $\CO \in \bar{\Nil}(\fg^*)$.
    \begin{enumerate}
        \item If $\nu \in {^a\fh^*} = \BC^n$ ($n = p+q$) is integral, that is, the differences of its coordinates are integral, then its coordinates can be permuted such that 
        \[ 
        \nu =  (\underbrace{\lambda_1, \cdots, \lambda_1}_{d_1}, \underbrace{\lambda_2, \cdots, \lambda_2}_{d_2}, \cdots, \underbrace{\lambda_k, \cdots, \lambda_k}_{d_k} ) \in \BC^n, 
        \]
        where $[d_1, d_2, \cdots, d_k]$ is a partition of $n$, and the $\lambda_i \in \BC$ satisfy the condition $\lambda_i - \lambda_j \in \BZ \backslash \{0\}$ for any $i \neq j$. Then
        \begin{equation}
            \sharp(\Irr_{\nu}(G;\CO)) = \left\{
            \begin{aligned}
                &\sharp\bigl(\mathrm{P}_{A}^{p,q}(\iota(\CO))\bigr) \cdot \sharp\bigl(\mathrm{A}_{[d_1,\cdots,d_{k}]}(\iota(\CO))\bigr), && \textrm{if $\lambda_1 \in \frac{n-1}{2} + \BZ$}, \\[2pt]
                &\sharp\bigl(\mathrm{A}_{[d_1,\cdots,d_k]}(\iota(\CO))\bigr)\delta_{p,q}, && \begin{tabular}[t]{@{}l@{}}\textrm{if $\lambda_1 \in \frac{n}{2} + \BZ$ and every}\\[1pt] \textrm{row length of $\iota(\CO)$ is even},\end{tabular} \\[5pt]
                &0, && \textrm{otherwise}.
            \end{aligned}
            \right.
        \end{equation}

        \item For an arbitrary $\nu \in {^{a}\fh^*}$, there exist irreducible Casselman--Wallach representations with this infinitesimal character only if the coordinates of $\nu$ can be permuted such that it has the form
        \[
            \nu = (\blam, \blam', \blam_1, \blam_1', \cdots, \blam_r, \blam_r') \in {^{a}\fh}^* = \BC^n \ (r \geq 0),
        \]
        where 
        \begin{align} 
            &\blam = (\underbrace{\lambda_1, \cdots, \lambda_1}_{d_1}, \cdots, \underbrace{\lambda_k, \cdots, \lambda_k}_{d_k}) \in \left(\frac{n-1}{2} + \BZ\right)^{e} \ (e \geq 0 ) ,\\
            &\blam' = (\underbrace{\lambda_1', \cdots, \lambda_1'}_{d_1'}, \cdots, \underbrace{\lambda_{k'}', \cdots, \lambda_{k'}'}_{d_{k'}'}) \in \left(\frac{n}{2} + \BZ\right)^{e'} (e' \geq 0 \  \textrm{is even}),\\
            &\blam_i = (\underbrace{\lambda_{i,1}, \cdots, \lambda_{i,1}}_{d_{i,1}}, \cdots, \underbrace{\lambda_{i,k_i}, \cdots, \lambda_{i,k_i}}_{d_{i,k_i}}) \in \BC^{e_i} \ (e_i \geq 1),\\
            &\blam_i' = (\underbrace{\lambda'_{i,1}, \cdots, \lambda'_{i,1}}_{d_{i,1}'}, \cdots, \underbrace{\lambda'_{i,k'_i}, \cdots, \lambda'_{i,k'_i}}_{d_{i,k_i}'}) \in \BC^{e_i},
        \end{align}
        and $[d_{i,1}, \cdots, d_{i,k_i}]$, $[d_{i,1}', \cdots, d_{i,k_i'}']$ are partitions of $e_i$, with the following conditions 
        \begin{itemize}
            \item each $\blam_i$ (resp. $\blam'_j$) is integral but $(\blam_i,\blam_j) \in \BC^{e_i+e_j}$ (resp. $(\blam'_i,\blam'_j) \in \BC^{e'_i+e'_j}$) is not integral for any $i \neq j$;
            \item $\lambda_{i,1}+\lambda_{j,1}\notin\BZ$ for all $1\le i<j\le r$;
            \item $\lambda_{i,a} + \lambda_{i,b}' \in \BZ$ and  $\lambda_{i,a} \notin \frac{1}{2}\BZ$ for any $i \in \{1, \cdots, r\}$ and any $a,b \in \{1,\cdots,k_i\}$;
            \item $\frac{e'}{2} + e_1 + e_2 + \cdots + e_r \leq \mathrm{min}\{p,q\}$.
        \end{itemize}
        In this case, let $p' = p - \left(\frac{e'}{2} + e_1 + \cdots + e_r\right) $, $q' = q - \left(\frac{e'}{2} + e_1 + \cdots + e_r\right)$, if $p+q$ is even, then
        \begin{align*}
            \sharp(\Irr_{\nu}(G;\CO)) = & \sum_{\substack{(\iota, \iota', \iota_1,\iota_1' \cdots,\iota_r, \iota_r') \in \mathrm{YD}_{e} \times \mathrm{YD}_{e'} \times \mathrm{YD}_{e_1} \times \mathrm{YD}_{e_1} \times \cdots \times \mathrm{YD}_{e_r} \times \mathrm{YD}_{e_r}  \\ \iota \mathop{\sqcup}\limits^r \iota' \mathop{\sqcup}\limits^r \iota_1 \mathop{\sqcup}\limits^r \iota_1 \mathop{\sqcup}\limits^r \cdots  \mathop{\sqcup}\limits^r \iota_r \mathop{\sqcup}\limits^r \iota_r   = \iota(\CO)}}  \sharp(\Irr_{\blam}(\U(p',q');\CO(\iota)))\\
            & \cdot \sharp\left(\Irr_{\blam'}\left(\U\left(\frac{e'}{2},\frac{e'}{2}\right);\CO(\iota')\right)\right)\cdot \prod_{i=1}^{r} \sharp\left(\Irr_{(\blam_i,-\blam_i')}(\GL_{e_i}(\BC);\CO(\iota_i)\times \CO(\iota_i'))\right),
        \end{align*}
        
        If $p+q$ is odd, then
        \begin{align*}
            \sharp(\Irr_{\nu}(G;\CO)) = & \sum_{\substack{(\iota, \iota', \iota_1,\iota_1' \cdots,\iota_r, \iota_r') \in \mathrm{YD}_{e} \times \mathrm{YD}_{e'} \times \mathrm{YD}_{e_1} \times \mathrm{YD}_{e_1} \times \cdots \times \mathrm{YD}_{e_r} \times \mathrm{YD}_{e_r}  \\ \iota \mathop{\sqcup}\limits^r \iota' \mathop{\sqcup}\limits^r \iota_1 \mathop{\sqcup}\limits^r \iota_1 \mathop{\sqcup}\limits^r \cdots  \mathop{\sqcup}\limits^r \iota_r \mathop{\sqcup}\limits^r \iota_r   = \iota(\CO)}} \sharp(\Irr_{\blam}(\U(p',q');\CO(\iota)))\\
            & \cdot \sharp\left(\Irr_{\blam' + \left(\frac{1}{2},\cdots,\frac{1}{2}\right)}\left(\U\left(\frac{e'}{2},\frac{e'}{2}\right);\CO(\iota')\right)\right)\cdot \prod_{i=1}^{r} \sharp\left(\Irr_{(\blam_i,-\blam_i')}(\GL_{e_i}(\BC);\CO(\iota_i)\times \CO(\iota_i'))\right).
        \end{align*}
    \end{enumerate}
\end{thm}

\begin{examp}
    Suppose $G = \U(2,1)$, and $\nu = (1,1,2) \in \BC^3$. Let $\CO_1$, $\CO_2$, and $\CO_3$ be the nilpotent orbits of $\fg$, whose Young diagrams are, respectively
    
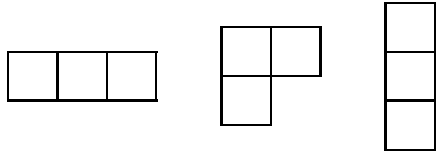
\[
        \begin{ytableau}
            ~&~&~
        \end{ytableau}
    \qquad
        \begin{ytableau}
            ~&~\\
            ~
        \end{ytableau}
    \qquad
        \begin{ytableau}
            ~\\
            ~\\
            ~
        \end{ytableau}
    \]
    the number of assignments of type $[2,1]$ on $\CO_1$, $\CO_2$, and $\CO_3$ are, respectively $1$, $1$, and $0$.

    On the first diagram, there is only one painting, given by
    
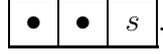
\[
        \begin{ytableau}
            \bullet & \bullet & s
        \end{ytableau}.
    \]

    On the second diagram, there are two paintings, given by
    
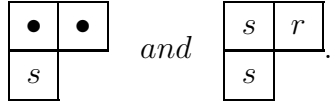
\[
    \begin{ytableau}
        \bullet & \bullet\\
        s    
    \end{ytableau}
    \quad and \quad
    \begin{ytableau}
        s & r\\
        s
    \end{ytableau}.
    \]
    
    So we can conclude that 
    \begin{itemize}
        \item $\sharp(\Irr_{\nu}(G;\CO_{1})) = 1$;
        \item $\sharp(\Irr_{\nu}(G;\CO_{2})) = 2$;
        \item $\sharp(\Irr_{\nu}(G;\CO_{3})) = 0$.
    \end{itemize}

\end{examp}

\subsection{Structure of this article}

In Section 2, we present the preliminaries for the counting formula, including coherent families, coherent continuation representations, and the representation theory of Weyl groups. In Section 3, we apply the theories developed in the previous sections to the specific cases of general linear groups over $\BR$, $\BC$, and $\BH$, and real unitary groups, obtaining explicit counting results.

\section{Preliminaries}\label{2}

\subsection{Associated varieties}\label{2.1}

Following the notations introduced in Section \ref{1.2} , denote by $(\CU(\fg)_{k})_{k = 0,1,2,\cdots}$ the natural increasing filtration of $\CU(\fg)$, where $\CU(\fg)_{k}$ is the subspace of $\CU(\fg)$ spanned by elements of the form $X_{1} X_{2} \cdots X_{m}$ for $0 \leq m \leq k$, with $X_{j} \in \fg$ for $1 \leq j \leq m$. By the Poincar\'e-Birkhoff-Witt theorem, we can identify the associated ring
\[
    \mathrm{gr}(\CU(\fg)) = \bigoplus_{i \geq 0}\CU(\fg)_{i+1} / \CU(\fg)_{i}
\]   
with the symmetric algebra $\mathrm{S}(\fg)$ in the canonical way.

For every two-sided ideal $I$ of $\CU(\fg)$, we define its associated variety $\mathrm{AV}(I)$ as the subvariety of $\fg^*$ annihilated by the ideal
\[
\mathrm{gr}(I) := \bigoplus_{i \geq 0} \frac{I \cap \CU(\fg)_{i+1}}{I \cap \CU(\fg)_i} \subseteq \RS(\fg),
\]
which is clearly stable under the action of $\Inn(\fg)$ on $\fg^*$. 

In particular, if $I$ is $\CZ(\fg)$-finite, that is, if the algebra $\CZ(\fg) / (I \cap \CZ(\fg))$ has finite dimension, then $\mathrm{AV}(I)$ is a Zariski closed subset of $\Nil(\fg^*)$ (see \cite[Corollary~5.3]{Vog91}).

If $V \in \Rep(G)$ is a Casselman--Wallach representation of $G$, then its annihilator 
$$\Ann(V) := \{X \in \CU(\fg) | X\cdot v = 0 \ \textrm{for any  $v \in V$} \} \subseteq \CU(\fg)$$
is automatically a $\CZ(\fg)$-finite two-sided ideal, we define the complex associated variety $\mathrm{AV}_{\mathbb{C}}(V)$ of $V$ as the associated variety of the ideal $\Ann(V)$, which is a $\Inn(\fg)$-stable closed subvariety of $\Nil(\fg^*)$.  Suppose $S$ is a $\Inn(\fg)$-stable Zariski closed subset of $\Nil(\fg^*)$, denote by $\Rep_S(G)$ the category of Casselman--Wallach representations of $G$ whose complex associated variety is contained in $S$.

By the work of Borho and Brylinski \cite{BB}, and Joseph \cite{Jos}, the associated variety of any primitive ideal (annihilator of an irreducible $\fg$-module) is the Zariski closure of a single nilpotent orbit. This implies the complex associated variety of any irreducible Casselman--Wallach representation is the Zariski closure of a single nilpotent orbit.

\subsection{Genuine representations of \texorpdfstring{$\tilde{\U}(p,q)$}{U(p,q)}}

In this subsection, we assume $G = \tilde{\U}(p,q)$. A representation of $\tilde{\U}(p,q)$ is called \emph{genuine} if the central subgroup $\{\pm 1\}$, which is the kernel of the covering homomorphism $\tilde{\U}(p,q) \to \U(p,q)$, acts on it via the nontrivial character.

We denote by $\Rep^{\mathrm{gen}}(\tilde{\U}(p,q))$ the full subcategory of genuine representations, and by $\Rep_{\nu}^{\mathrm{gen}}(\tilde{\U}(p,q))$ the full subcategory of genuine representations with infinitesimal character $\nu$. Their Grothendieck groups and sets of irreducible objects are then
\[
\CK^{\mathrm{gen}}(\tilde{\U}(p,q)),\quad \CK_{\nu}^{\mathrm{gen}}(\tilde{\U}(p,q)), \quad 
\Irr^{\mathrm{gen}}(\tilde{\U}(p,q)), \quad \text{and}\quad \Irr_{\nu}^{\mathrm{gen}}(\tilde{\U}(p,q)),
\]
respectively.

Let $\Irr_{\nu}^{\mathrm{gen}}\left(\tilde{\U}(p,q);\CO\right)$ denote the subset of $\Irr_{\nu}^{\mathrm{gen}}\left(\tilde{\U}(p,q)\right)$ consisting of representations with complex associated variety $\bar{\CO}$.

Note that there is a bijection
\defmap{\psi}{\Irr_{\nu}^{\mathrm{gen}}\bigl(\tilde{\U}(p,q);\CO\bigr)}%
{\Irr_{\nu + \mu}\bigl(\U(p,q);\CO\bigr)}{V}{V \otimes \mathrm{det}^{\frac{1}{2}}}
where $\mu \in {^{a}\fh^*}$ is the differential of the character $\mathrm{det}^{\frac{1}{2}}$. 

Consequently, the number of genuine irreducible representations of $\tilde{\U}(p,q)$ with infinitesimal character $\nu$ and complex associated variety $\bar{\CO}$ coincides with the number of irreducible representations of $\U(p,q)$ with infinitesimal character $\nu + \mu$ and complex associated variety $\bar{\CO}$.

\subsection{Generalities on coherent families of Casselman--Wallach representations}\label{2.2}
This section is devoted to the introduction of coherent continuation representations, following the exposition in \cite{BMSZ}.

Let $G$ be a real reductive group in Harish-Chandra's class. Fix a connected complex reductive Lie group $G_\BC$, together with a homomorphism of real Lie groups $\iota: G \to G_\BC$ such that the differential has the following properties:
\begin{itemize}
   \item the kernel of $\mathrm{d}\iota$ is contained in the center of $\mathrm{Lie}(G)$;
   \item the image of $\mathrm{d}\iota$ is a real form of $\mathrm{Lie}(G_\BC)$.
\end{itemize}
It can be verified directly from the above assumption that the homomorphism $\iota$ induces a unique morphism between the abstract Cartan subalgebra of $\Lie(G)$ and that of $\Lie(G_{\BC})$. The analytic weight lattice $Q_{an}$ of $G_\BC$ is identified with a subgroup of ${^{a}\fh^*}$ via this morphism, and is denoted by $Q_{\iota} \subseteq  {^{a}\fh^*}$. Moreover, $Q_{\iota}$ is $W$-stable and $Q_{\fg} \subseteq Q_{\iota} \subseteq Q^{\fg}$. In the rest of this section, we fix a $Q_{\iota}$-coset $\Lambda = \lambda + Q_{\iota} \subseteq {^a\fh^*}$. Let $W_{\Lambda}$ denote the stabilizer of $\Lambda$ in $W$, namely, 
$$W_{\Lambda} := \set{w \in W}{w\lambda - \lambda \in Q_{\iota}}.$$

We also put 
\begin{equation}
    \Delta(\Lambda) := \set{\alpha \in \Delta}{\textrm{$\langle \check{\alpha} , \nu \rangle \in \BZ$ for some (and all) $\nu \in \Lambda$}}.
\end{equation}
This is a root system with the corresponding coroots 
\begin{equation}
    \check{\Delta}(\Lambda) := \set{\check{\alpha} \in \check{\Delta}}{\textrm{$\langle \check{\alpha}, \nu \rangle \in \BZ$ for some (and all) $\alpha \in \Delta$}}.
\end{equation}
Let $W(\Lambda) \subseteq W$ denote the Weyl group of the root system $\Delta(\Lambda)$, which is referred to as the integral Weyl group associated with $\Lambda$.

Let $\Rep_{\mathrm{hol}}(G_{\BC})$ be the category of finite-dimensional holomorphic representations of $G_{\BC}$, its Grothendieck ring is denoted by $\CR_{\mathrm{hol}}(G_{\BC})$. Then $\CK(G)$ is a module over $\CR_{\mathrm{hol}}(G_{\BC})$ under the action of tensor product.

\begin{defn}
   A $\CK(G)$-valued $\Lambda$-coherent family is a map
   $$\Phi: \Lambda \to \CK(G),$$
   such that:
   \begin{itemize}
      \item for any $\nu \in \Lambda$, $\Phi(\nu) \in \CK_{\nu}(G)$,
      \item for any $F \in \CR_{\mathrm{hol}}(G_{\BC})$ and $\nu \in \Lambda$, $F \cdot (\Phi(\nu)) = \sum_{\mu \in \Delta(F)} \Phi(\nu + \mu)$ (where $\Delta(F)$ is the set of weights of $F$ counted multiplicity).
   \end{itemize}
\end{defn}

Let $\mathrm{Coh}_{\Lambda}(\CK(G))$ denote the complex vector space of all coherent families on $\Lambda$. It is a representation of $W_{\Lambda}$ under the action
$$(w \cdot \Psi)(\nu) = \Psi(w^{-1}\nu), \ \textrm{for any $w \in W_{\Lambda}$, $\Psi \in \mathrm{Coh}_{\Lambda}(\CK(G))$, $\nu \in \Lambda$.}$$
This is called the \emph{coherent continuation representation}.

Note that all definitions remain valid if we replace $\CK(G)$ by $\CK^{\mathrm{gen}}(G)$ in the $\tilde{A}$ case.

We then introduce the notion of parameters for coherent continuation representations and give an explicit formula to compute them in terms of these parameters.

Suppose $H$ is a Cartan subgroup of $G$, its complexified Lie algebra is denoted by $\fh$, which is a Cartan subalgebra of $\fg$. $H$ has a unique maximal compact subgroup $T$. Denote by $\Delta_{\fh} \subseteq \fh^*$ the root system of $\fg$. A root is called imaginary if $\check{\alpha} \in \ft$, where $\ft = \mathrm{Lie}(T)$, or equivalently if it takes purely imaginary values on $\fh^*$. Moreover, an imaginary root $\alpha$ is called compact imaginary if the corresponding root space $\fg_{\alpha}$ is contained in a complexified Lie algebra of a compact subgroup of $G$.

There are two fundamental facts about the representation theory of $H$:
\begin{itemize}
    \item Every irreducible Casselman--Wallach representation of $H$ is finite-dimensional.
    \item For every $\Gamma \in \mathrm{Irr}(H)$, the differential of $\mathrm{\Gamma}$ is a direct sum of one-dimensional representations attached to a unique $\mathrm{d\Gamma} \in \fh^*$.
\end{itemize}

For every Borel subalgebra $\fb$ of $\fg$ containing $\fh$, write
$$\xi_{\fb}: \fh \to {^{a}\fh},$$
for the linear isomorphism attached to $\fb$ defined by
$$\fh \hookrightarrow \fb \to \fb / [\fb , \fb] =  {^{a}\fh},$$
the transpose inverse of this map is still denoted by $\xi_{\fb}: {^{a}\fh^*} \to \fh^*$.

Write
\begin{equation}
    W(^{a}\fh^*,\fh^*) = \set{\xi_{\fb}: {^{a}\fh}^* \to \fh^*}{\textrm{$\fb$ is a Borel subalgebra containing $\fh$}},
\end{equation}
put
\begin{equation}
    \delta(\xi) := \frac{1}{2} \cdot \sum_{\textrm{$\alpha$ is an imaginary root in $\xi \Delta^{+}$}}\alpha - \sum_{\textrm{$\beta$ is a compact imaginary root in $\xi \Delta^{+}$}}\beta \in \fh^*.
\end{equation}

\begin{defn}\label{regular character}
    Write $\mathscr{P}_{\Lambda}(G)$ for the set of all triples $\upgamma = (H,\xi,\mathrm{\Gamma})$, where $H$ is a Cartan subgroup of $G$, $\xi \in W(^{a}\fh^*,\fh^*)$, and
    \defmap{\Gamma}{\Lambda}{\Irr(H)}{\nu}{\Gamma_{\nu}}
    is a map with the following properties:
    \begin{itemize}
        \item $\Gamma_{\nu+\beta} = \Gamma_{\nu} \otimes \xi(\beta)$ for all $\beta \in Q_{\iota}$ and $\nu \in \Lambda$;
        \item $\mathrm{d}\Gamma_{\nu} = \xi(\nu) + \delta(\xi)$ for all $\nu \in \Lambda$.
    \end{itemize}
    Here $\xi(\beta)$ is viewed as a character of $H$ via the homomorphism $\iota: H \to H_{\BC}$, $H_{\BC}$ is the unique Cartan subgroup of $G_{\BC}$ containing $\iota(H)$.
\end{defn}

The group $G$ acts on $\mathscr{P}_{\Lambda}(G)$ in the standard way, and we define the set of parameters for $\mathrm{Coh}_{\Lambda}(G)$ to be
\begin{equation}
    \CP_{\Lambda}(G) := G \backslash \mathscr{P}_{\Lambda}(G) .
\end{equation}

For each $\gamma \in \CP_{\Lambda}(G)$, represented by $\upgamma = (H,\xi,\Gamma)$, by \cite[Theorem 8.2.1]{Vog81}, we have two $\CK(G)$-valued coherent families $\Psi_{\gamma}$ and $\bar{\Psi}_{\gamma}$ on $\Lambda$ such that
$$\Psi_{\gamma}(\nu) = X(\Gamma_{\nu},\xi(\nu)) \ \textrm{and}  \ \bar{X}_{\gamma}(\nu) = \bar{X}(\Gamma_{\nu},\xi(\nu))$$
for all regular dominant element $\nu \in \Lambda$. Here $X(\Gamma_{\nu},\xi(\nu))$ is the standard representation defined in \cite[Notation Convention 6.6.3]{Vog81} and $\bar{X}_{\gamma}(\nu)$ is its unique irreducible subrepresentation (see \cite[Theorem 6.5.12]{Vog81}).

By Langlands classification, $\{ \bar{\Psi}_{\gamma} \}_{\gamma \in \CP_{\Lambda}(G)}$ is a basis of $\mathrm{Coh}_{\Lambda}(\CK(G))$ (\cite[Proposition 6.6.7]{Vog81}), and view $\mathrm{Coh}_{\Lambda}(\CK(G))$ as a basel representation of $W_{\Lambda}$ with this basis  (basel representation means a representation with a fixed basis). The family $\{ \Psi_{\gamma} \}_{\gamma \in \CP_{\Lambda}(G)}$ is also a basis of $\mathrm{Coh}_{\Lambda}(\CK(G))$ (\cite[Theorem 6.5.12]{Vog81}).

We now demonstrate how to compute coherent continuation representations using the parameters defined above.
The cross action of $W_{\Lambda}$ on the set $\mathscr{P}_{\Lambda}(G)$ is defined by (\cite[Definition 4.2]{Vog82}):
\begin{equation}
    w \times (H, \xi, \Gamma) = (H, \xi \circ w^{-1}, (\nu \mapsto \Gamma_{\nu} \otimes (\xi \circ w^{-1}(\nu) + \delta(\xi \circ w^{-1}) - \xi (\nu) - \delta(\xi)))).
\end{equation}

This commutes with the action of $G$ and thus descends to an action on $\CP_{\Lambda}(G)$:
\begin{equation}
    W_{\Lambda} \times \CP_{\Lambda}(G) \to \CP_{\Lambda}(G), \ (w,\gamma) \mapsto w \times \gamma.
\end{equation}

Since $G$ is in Harish-Chandra's class, the real Weyl group
\begin{equation}
    W_{H} := N_{G}(H)/H
\end{equation}
is identified with a subgroup of the complex Weyl group $W_{\fh}$ of $\fg$ with respect to $\fh$.

We have an inclusion:
\begin{equation}
    W(\Delta_{\fh,\mathrm{im}}) \hookrightarrow  W_{\fh,\ft}
\end{equation}
where $W(\Delta_{\mathrm{\fh,im}})$ is the Weyl group for the imaginary root system $\Delta_{\fh,\mathrm{im}}$. There is a group action of $W_{\fh,\ft}$ on the set of positive systems of $\Delta_{\fh,\mathrm{im}}$, with the subgroup $W(\Delta_{\fh,\mathrm{im}})$ acts simply transitively. It is easy to verify that for any choice of positive systems $\Delta_{\fh,\mathrm{im}}^{+}$ there is a natural decomposition:
\begin{equation}
    W_{\fh,\ft} = W_{\fh,\ft,\Delta_{\fh,\mathrm{im}}^{+}} \ltimes W(\Delta_{\fh,\mathrm{im}}),
\end{equation}
where $W_{\fh,\ft,\Delta_{\fh,\mathrm{im}}^{+}}$ is the stabilizer of $\Delta_{\fh,\mathrm{im}}^{+}$ in $W_{\fh,\ft}$. For different choice of $\Delta_{\fh,\mathrm{im}}^{+}$, the subgroups $W_{\fh,\ft,\Delta_{\fh,\mathrm{im}}^{+}}$ conjugate to each other. So there is a unique quartic character
$$\mathrm{sgn_{im}}: W_{\fh,\ft} \to \BC^{\times}$$
such that
\begin{itemize}
    \item its restriction to $W(\Delta_{\fh,\mathrm{im}})$ equals the sign character;
    \item its restriction to $W_{\fh,\ft,\Delta_{\fh,\mathrm{im}}^{+}}$ is trivial for some (and hence all) positive system $\Delta_{\fh,\mathrm{im}}^{+}$ of $\Delta_{\fh,\mathrm{im}}$.
\end{itemize}

Fix a $\gamma \in \CP_{\Lambda}(G)$, and choose a representative $(H,\xi,\Gamma) \in \mathscr{P}_{\Lambda}(G)$ for it. Denote by $W_{\gamma} \subseteq W_{\Lambda}$ the stabilizer of $\gamma$ under the cross action, and by $W_{\fh,\ft}$ the stabilizer of $\ft$ in $W_{\fh}$, then there are inclusions
$$\xi \circ W_{\gamma} \circ \xi^{-1} \subseteq W_{H} \hookrightarrow W_{\fh,\ft} \hookrightarrow W_{\fh}.$$

Therefore, we obtain a quartic character on $W_{\gamma}$:
\defmap{\mathrm{sgn_{\gamma}}}{W_{\gamma}}{\BC^{\times}}{w}{\mathrm{sgn_{im}}(\xi \circ w \circ \xi^{-1})}
This character is independent of the choice of the representative $(H,\xi,\Gamma)$.

The coherent continuation representation can be computed using the basis of standard modules $\{ \Psi_{\gamma} \}_{\gamma \in \CP_{\Lambda}(G)}$ when the group $G$ is linear (see \cite[Section14]{Vog82}). The following result is due to Barbasch-Vogan, in a suitably modified form from \cite[Proposition 2.4]{BV83}.

\begin{thm}\label{Coh}(cf. \cite[Theorem 4.5]{BMSZ}) 
Assume that $G$ is linear. As a representation of $W_{\Lambda}$,
    \begin{equation}
        \mathrm{Coh}_{\Lambda}(\CK(G)) \cong \bigoplus_{\gamma} \mathrm{Ind}_{W_{\gamma}}^{W_{\Lambda}} \sgn_{\gamma}
    \end{equation}
    where $\gamma$ runs over a representative set of the $W_{\Lambda}$-orbits in $\CP_{\Lambda}(G)$ under the cross action.
\end{thm}

We also mention that the coherent continuation representation has already been computed in many cases in \cite[Section 7-8]{BMSZ}, \cite{BMSZ24}, and \cite{McG}.

\subsection{Representation theory of Weyl groups}\label{2.3}
Let $H$ be a finite group with a linear action on a complex vector space $V$, $H_1 \subseteq H$ be a subgroup, and denote by $V^{H_1}$ the subspace of $V$ defined by
$$V^{H_1} := \set{v \in V}{hv = v, \textrm{for all $h \in H_1$}}$$
it is an $H_1$-submodule of $V$, so there is a decomposition $V = V^{H_1} \oplus V_1$, where $V_1$ is an $H_1$-submodule which has no nonzero $H_1$-invariants.

\begin{thm}[Macdonald, Lusztig, and Spaltenstein]\label{j-ind}
    Denote by $\mathrm{S}^e(V_1)$ the $e$-th symmetric power of $V_1$. Suppose $\sigma_1 \in \Irr(H_1)$ is an irreducible representation that occurs with multiplicity $1$ in $\mathrm{S}^e(V_1)$ and does not occur in $\mathrm{S}^{i}(V_1)$ if $0 \leq i \leq e-1$. We may regard $\sigma_1$ as a subspace of $\mathrm{S}^e(V)$ via the inclusion $\sigma_1 \subseteq \BC \otimes \mathrm{S}^e(V_1) \subseteq \mathrm{S}^e(V)$, and consider the $H$-representation $\sigma$ of $\mathrm{S}^e(V)$ generated by $\sigma'$. Then
    \begin{enumerate}
        \item $\sigma$ is an irreducible $H$-representation;
        \item $\sigma$ occurs with multiplicity $1$ in $\mathrm{S}^e(V)$;
        \item $\sigma$ does not occur in $\mathrm{S}^{i}(V)$ if $0 \leq i \leq e-1$.
    \end{enumerate}
\end{thm}

\begin{proof}
    See \cite[Theorem 11.2.1]{Car}.
\end{proof}

Let $\Irr(H_1,e)$ denote the subset of all $\sigma \in \Irr(H_1)$ such that $\sigma$ occurs with multiplicity $1$ in $\mathrm{S}^e(V')$ and does not occur in $\mathrm{S}^{i}(V')$ if $0 \leq i \leq e-1$.  Then by the theorem above, we have a map called $j$-induction
\defmap{j_{H_1}^H}{\Irr(H_1,e)}{\Irr(H,e)}{\sigma}{j_{H_1}^{H}(\sigma)}
where $j_{H_1}^{H}(\sigma)$ is the subrepresentation of $\mathrm{S}^e(V)$ generated by $\sigma$.

The following are some basic facts about $j$-induction.

\begin{prop}\label{2.5}
    Let $H_2 \subseteq H_1 \subseteq H$ be subgroups. Suppose that
    \begin{equation}
        \begin{split}
            V = V_1 \oplus V^{H_1}, \ & \textrm{where $V_1$ is an $H_1$-submodule}; \\
            V_1 = V_2 \oplus V_1^{H_2}, \ & \textrm{where $V_2$ is an $H_2$-submodule}.
        \end{split}
    \end{equation}
    Then $V = V_2 \oplus V^{H_2}$, and we have
    \begin{equation}
        j_{H_2}^{H} = j_{H_1}^{H} \circ j_{H_2}^{H_1}.
    \end{equation}
\end{prop}

\begin{proof}
    See \cite[Proposition 11.2.4]{Car}.
\end{proof}

\begin{prop}\label{2.6}
    Let $H$ and $H'$ be finite groups, with linear actions on complex vector spaces $V$, $V'$ respectively, $H_1 \subseteq H$ and $H'_1 \subseteq H'$ be subgroups. Let $\sigma_1 \in \Irr(H_1,e_1)$ and $\sigma'_1 \in \Irr(H'_1,e'_1)$, then
    \begin{enumerate}
        \item $\sigma_1 \boxtimes \sigma'_1 \in \Irr(H_1 \times H'_1,e_1 + e'_1)$;
        \item we have the following equation 
            \begin{equation}
            j_{H_1 \times H'_1}^{H \times H'}(\sigma_1 \boxtimes \sigma'_1) = j_{H_1}^{H}\sigma_1 \boxtimes j_{H'_1}^{H'}\sigma'_1,
            \end{equation}
            where $j_{H_1 \times H'_1}^{H \times H'}$ is defined by the representation $V \oplus V'$ of $H \times H'$.
    \end{enumerate}
\end{prop}

\begin{proof}
    This is well known, we provide a sketch of proof for completeness.
    Since
    $$\mathrm{S}^{n}(V_1 \oplus V'_1) = \bigoplus_{i = 0}^{n}\mathrm{S}^{i}(V_1) \otimes \mathrm{S}^{n-i}(V'_1),$$
    $\sigma_1 \boxtimes \sigma'_1 $ does not occur in $\mathrm{S}^{n}(V_1 \oplus V'_1)$ if $n \leq e_1 + e_2$, and in $\mathrm{S}^{e_1+e'_1}(V_1 \oplus V'_1)$,
    only $\mathrm{S}^{e_1}(V_1) \boxtimes \mathrm{S}^{e'_1}(V'_1)$ contains $\sigma_1 \boxtimes \sigma'_1$ as subrepresentation, the multiplicity is $1$ by assumption, so $\sigma_1 \boxtimes \sigma'_1 \in \Irr(H_1 \times H'_1, e_1+e'_1)$.
    The second statement can be verified directly.
\end{proof}

Back to our setting of real reductive groups, let $W \subseteq \GL({^{a}\fh})$ denote the abstract Weyl group. For every $\sigma \in \Irr(W)$, its fake degree is defined as 
\begin{equation}
    a(\sigma) := \mathrm{min}\set{k \in \BN}{\textrm{$\sigma$ occurs in the k-th symmetric power $\mathrm{S}^{k}({^{a}\fh})$}}.
\end{equation}
This is well-defined since every $\sigma \in \Irr(W)$ occurs in the symmetric algebra $\mathrm{S}(^{a}\fh)$. The representation is called univalent if it occurs in $\mathrm{S}^{a(\sigma)}(^{a}\fh_{s})$ with multiplicity one, where $^{a}\fh_{s}:= \mathrm{Span}(\check{\Delta})$ denotes the span of the coroots.

Recall Lusztig's notion of a special representation of a Weyl group \cite{Lus79}. An irreducible representation of $W$ is called Springer if it corresponds to the trivial local system on a nilpotent orbit in $\fg^*$ via the Springer correspondence  \cite{Spr}. Note that every special irreducible representation is Springer, and the corresponding nilpotent orbit is called a special nilpotent orbit. Every Springer representation is univalent \cite{BM}.

We have a decomposition
$$^{a}\fh = (\Delta(\Lambda))^{\perp} \oplus \mathrm{Span}\left(\check{\Delta}(\Lambda)\right)$$
where
$$(\Delta(\Lambda))^{\perp} := \set{x \in {^{a}\fh}}{\textrm{$\langle x , \alpha \rangle = 0$ for all $\alpha \in \Delta(\Lambda)$}}.$$
Then $(\Delta(\Lambda))^{\perp} = {^{a}\fh}^{W(\Lambda)}$. For every univalent irreducible representation $\sigma_0$ of $W(\Lambda)$, we view it as a subrepresentation of $\mathrm{S}^{a(\sigma_0)}({^{a}\fh})$ via the inclusions
$$\sigma_0 = \BC \otimes \sigma_0 \subseteq \mathrm{S}^0\left((\Delta(\Lambda))^{\perp}\right)\otimes \mathrm{S}^{a(\sigma_0)}\left(\mathrm{Span}\left(\check{\Delta}(\Lambda)\right)\right) \subseteq \mathrm{S}^{a(\sigma_0)}\left({^{a}\fh}\right)$$
The $W$-subrepresentation of $\mathrm{S}^{a(\sigma_0)}({^{a}\fh})$ generated by $\sigma_0$ is just $j_{W(\Lambda)}^{W}(\sigma_0)$, it is irreducible and univalent by Theorem \ref{j-ind}, with the same fake degree as that of $\sigma_0$.

If $\sigma_0$ is special, then the $j$-induction $j_{W(\Lambda)}^{W}(\sigma_0)$ is Springer. Write
\begin{equation}
    \CO_{\sigma_0} \in \bar{\mathrm{Nil}}(\fg^*) \textrm{for the nilpotent orbit corresponding to $j_{W(\Lambda)}^{W}(\sigma_0)$},
\end{equation}

There is also an equivalence relation $\approx$ on $\Irr(W(\Lambda))$, which depends only on $W(\Lambda)$ as an abstract Coxeter group; for a precise definition, see \cite[Section 3]{BMSZ}. An equivalence class of this equivalence relation is called a double cell in $\Irr(W(\Lambda))$. Note that this definition of double cell in $\Irr(W(\Lambda))$ coincides with that of Lusztig in \cite{Lus82}, and each double cell contains a unique special representation.

Denote by $\Irr^{\mathrm{sp}}(W(\Lambda))$ the set of special irreducible representations of $W(\Lambda)$.

\begin{defn}
    For an $\Ad(\fg)$-stable closed subset of $\fg^*$, define
    $$\Irr^{\mathrm{sp}}_{S}(W(\Lambda)) := \set{\sigma_0 \in \Irr^{\mathrm{sp}}(W(\Lambda))}{\CO_{\sigma_0} \subseteq S},$$
    and
    $$\Irr_{S}(W(\Lambda)) := \set{\sigma \in \Irr(W(\Lambda))}{\textrm{there is a $\sigma_0 \in \Irr^{\mathrm{sp}}_S(W(\Lambda))$ such that $\sigma \approx \sigma_0$}}.$$

    For a nilpotent orbit $\CO \in \bar{\Nil}(\fg^*)$, define
    $$\Irr^{\mathrm{sp}}(W(\Lambda);\CO) := \set{\sigma_0 \in \Irr^{\mathrm{sp}}(W(\Lambda))}{\CO_{\sigma_0} = \CO},$$
    and
    $$\Irr(W(\Lambda);\CO) := \set{\sigma \in \Irr(W(\Lambda))}{\textrm{there is a $\sigma_0 \in \Irr^{\mathrm{sp}}(W(\Lambda;\CO))$ such that $\sigma \approx \sigma_0$}}.$$
\end{defn}

It follows directly from the definition that
\begin{equation}
    \Irr_{S}^{\mathrm{sp}}(W(\Lambda)) = \bigsqcup_{\CO \subseteq S} \Irr^{\mathrm{sp}}(W(\Lambda);\CO),
\end{equation}
and 
\begin{equation}\label{(2.16)}
    \Irr_{S}(W(\Lambda)) = \bigsqcup_{\CO \subseteq S}\Irr(W(\Lambda);\CO).
\end{equation}

In the rest of this section, we mainly focus on the Weyl groups of type $A$, in which case $W = \mathrm{S}_n$ is the symmetric group, which has a linear action on the $n$-dimensional complex vector space $^{a}\fh = \BC^n$. Then, all irreducible representations are special, and each double cell is a singleton (see \cite{Lus79} and \cite{Lus82}).

The following is a construction of irreducible representations of the symmetric group $\mathrm{S}_n$, see \cite[Chapter 5.4]{GP} for a proof.

\begin{thm}
    There is a bijection between $\mathrm{YD}_n$ and $\Irr(\mathrm{S}_n)$:
    \defmap{\phi}{\mathrm{YD}_n}{\Irr(\mathrm{S}_n)}{\iota}{j_{\mathrm{S}_{\iota^t}}^{\mathrm{S}_n}\sgn_{\iota^t}}
    where $\iota^t$ is the transpose young diagram of $\iota$, $\mathrm{S}_{\iota^t} := \mathrm{S}_{\br_1(\iota^t)} \times \cdots \times \mathrm{S}_{\br_k(\iota^t)}$ ($k$ is the number of rows of $\iota^t$) viewed as a subgroup of $\mathrm{S}_n$, $\sgn_{\iota^t} := \sgn_{\br_1(\iota^t)} \times \cdots \times \sgn_{\br_{k}(\iota^t)}$ is the sign character of $\mathrm{S}_{\iota^t}$, and the $j$-induction is defined via the standard representation of $\mathrm{S}_n$ on $^{a}\fh$.
\end{thm}

When we identify nilpotent orbits with partitions of $n$, the bijection $\phi$ coincides with the Springer correspondence between nilpotent orbits and irreducible representations of Weyl groups.

Using the bijection $\phi$, we can describe the $j$-induction of symmetric groups specifically.

\begin{prop}\label{2.9}
    Let $\iota = [d_1,d_2,\cdots,d_r]$ be a partition of $n$, and for each $i \in \{1,\cdots,r\}$, let $\iota_i \in \mathrm{YD}_{d_i}$, then
    \begin{equation}
        j_{\mathrm{S}_1 \times \cdots \mathrm{S}_{d_r}}^{\mathrm{S}_n}\phi(\iota_1) \boxtimes \cdots \boxtimes \phi(\iota_r) = \phi\left(\iota_1 \mathop{\sqcup}\limits^r \cdots \mathop{\sqcup}\limits^r \iota_r\right).
    \end{equation}
\end{prop}

\begin{proof}
    We indicate the idea of the proof for the convenience of the reader.
    According to Proposition \ref{2.6}, the $j$-induction compatible with direct product, we have
    \begin{align}
        \phi(\iota_1) \boxtimes \cdots \boxtimes \phi(\iota_r) & = j^{\mathrm{S}_{d_1}}_{\mathrm{S}_{\iota_{1}^t}}\sgn_{\iota_1^t} \boxtimes \cdots \boxtimes j^{\mathrm{S}_{d_r}}_{\mathrm{S}_{\iota_{r}^t}}\sgn_{\iota_r^t}\\
        & = j_{\mathrm{S}_{\iota_{1}^t} \times \cdots \times \mathrm{S}_{\iota_{r}^t}}^{\mathrm{S}_{d_1} \times \cdots \times \mathrm{S}_{d_r}}\sgn_{\iota_{1}^t} \boxtimes \cdots \boxtimes \sgn_{\iota_{r}^t}.
    \end{align}
    Then use Proposition \ref{2.5}
    \begin{align}
         j_{\mathrm{S}_1 \times \cdots \mathrm{S}_{d_r}}^{\mathrm{S}_n}\phi(\iota_1) \boxtimes \cdots \boxtimes \phi(\iota_r) & = j_{\mathrm{S}_1 \times \cdots \mathrm{S}_{d_r}}^{\mathrm{S}_n}\left(j_{\mathrm{S}_{\iota_{1}^t} \times \cdots \times \mathrm{S}_{\iota_{r}^t}}^{\mathrm{S}_{d_1} \times \cdots \times \mathrm{S}_{d_r}}\sgn_{\iota_{1}^t} \boxtimes \cdots \boxtimes \sgn_{\iota_{r}^t}\right)\\
         & = j^{\mathrm{S}_n}_{\mathrm{S}_{\iota_{1}^t} \times \cdots \times \mathrm{S}_{\iota_{r}^t}}\sgn_{\iota_{1}^t} \boxtimes \cdots \boxtimes \sgn_{\iota_{r}^t}\\
         & = \phi\left(\iota_1 \mathop{\sqcup}\limits^r \iota_2 \cdots \mathop{\sqcup}\limits^r  \iota_r\right),
    \end{align}
    where the third equality is the definition of $\phi$.
\end{proof}

We end this section with some branching formulas for the induction of representations of symmetry groups. 

Let $\mathrm{W}_n = \mathrm{S}_{n} \ltimes \{\pm 1\}^n$ be the Weyl group of type $B/C$, and define quadratic character 
\defmap{\epsilon}{\mathrm{W}_n}{\{\pm 1\}}{(s,(x_1, \cdots, x_n))}{x_1\cdots x_n}
and as always, $\sgn$ denote the sign character.

\begin{prop}[Pieri's rule]\label{Pieri}
    Let $k + l = n$ be positive integers, $\iota \in \mathrm{YD}_{k}$ then
    \begin{equation}
        \Ind_{\mathrm{S}_k \times \mathrm{S}_l}^{\mathrm{S}_{n}}\phi(\iota) \boxtimes 1 = \bigoplus_{\nu} \phi(\nu),
    \end{equation}
    where the sum is over all Young diagram $\nu \in \mathrm{YD}_n$, which is obtained from adding $l$ boxes to $\iota$, with no two boxes in the same column.
    
    Similarly, we have
    \begin{equation}
        \Ind_{\mathrm{S}_{k} \times \mathrm{S}_{l}}^{\mathrm{S}_n}\phi(\iota) \boxtimes \sgn = \bigoplus_{\nu}\phi(\nu),
    \end{equation}
    where the sum is over all Young diagrams $\nu \in \mathrm{YD}_n$, which is obtained from adding $l$ boxes to $\iota$, with no two boxes in the same row.
\end{prop}

\begin{proof}
    See \cite[Corollary 6.1.7]{GP}.
\end{proof}

\begin{prop}\label{branch}
    If we identify $\mathrm{W_n}$ as a subgroup of $\mathrm{S}_{2n}$ in the natural way, then
    \begin{align}
        &\Ind_{\mathrm{W}_n}^{\mathrm{S}_{2n}} \epsilon = \bigoplus_{\substack{\sigma \in \mathrm{YD}_{2n}\\ \textrm{ $\bc_{i}(\sigma)$ is even for all $i \in \BN^+$}}} \phi(\sigma),\\
        &\Ind_{\mathrm{W}_n}^{\mathrm{S}_{2n}} 1 = \bigoplus_{\substack{\sigma \in \mathrm{YD}_{2n}\\ \textrm{ $\br_{i}(\sigma)$ is even for all $i \in \BN^+$}}} \phi(\sigma)
    \end{align}
\end{prop}

\begin{proof}
    See \cite[Lemma 4.1 (b)]{BV83}.
\end{proof}

\subsection{Counting formula}

The following inequality is the main tool for counting irreducible representations, as proved in \cite{BMSZ}. 

\begin{thm}\label{counting}
    Let $G$ be a real reductive group in Harish-Chandra's class, we have an inequality
    \begin{equation}
        \sharp(\Irr_{\nu,S}(G)) \leq \sum_{\sigma \in \Irr_{S}(W(\Lambda))} [1_{W_\nu}:\sigma] \cdot [\sigma:\mathrm{Coh}_{\Lambda}(\CK(G))],
    \end{equation}
    where $1_{W_\nu}$ denotes the trivial representation of the stabilizer $W_\nu$ of $\nu$ in $W$. The equality holds if the Coxeter group $W(\Lambda)$ has no simple factor of type $F_4$, $E_6$, $E_7$, or $E_8$, and $G$ is linear or isomorphic to a real metaplectic group.
\end{thm}

Moreover, using induction on the dimension of orbits together with Equation (\ref{(1.3)}) and Equation (\ref{(2.16)}), we obtain

\begin{thm}
    Under the setting of Theorem \ref{counting}, we have an inequality
    \begin{equation}
        \sharp(\Irr_{\nu}(G;\CO)) \leq \sum_{\sigma \in \Irr(W(\Lambda);\CO)}[1_{W_\nu}:\sigma] \cdot [\sigma:\mathrm{Coh}_{\Lambda}(\CK(G))]
    \end{equation}
    where $\CO \in \bar{\Nil}(\fg^*)$ is a nilpotent orbit.  The equality holds if the Coxeter group $W(\Lambda)$ has no simple factor of type $F_4$, $E_6$, $E_7$, or $E_8$, and $G$ is linear or isomorphic to a real metaplectic group.
\end{thm}

For the case $\tilde{A}$, the above theorem yields the following formula immediately.

\begin{cor}
    \begin{equation}
        \sharp\left(\Irr_{\nu}^{\mathrm{gen}}\left(\tilde{\U}(p,q);\CO\right)\right) = \sum_{\sigma \in \Irr(W(\Lambda);\CO)} [1_{W_\nu}:\sigma] \cdot [\sigma:\mathrm{Coh}_{\Lambda}(\CK^{\mathrm{gen}}(G))].
    \end{equation}
\end{cor}

\section{Counting results for general linear groups and unitary groups}

   \subsection{Standard representations of classical groups}\label{3.1}

   We introduce some notions in \cite{BMSZ} related to standard representations of classical groups. Given a label $\star$, define the notion of a $\star$-structure, which consists of a finite-dimensional complex vector space $V$ and some additional data. The group $\G(V)$ will be defined as the (real) subgroup of $\GL(V)$ fixing the $\star$-structure. Its Zariski closure in $\GL(V)$ will be denoted by $\G_{\BC}(V)$. The natural inclusion $\iota: \G(V) \hookrightarrow \G_{\BC}(V)$ satisfies the assumptions at the beginning of Section \ref{2.1}.

    \textbf{The case when $\star \in \{A^{\BR},A^{\BH}\}$}. In this case, a $\star$-structure consists of a finite dimensional complex vector space $V$, and a conjugate linear automorphism $\mathbf{j}: V \to V$ such that
   $$\mathbf{j}^2 = \left\{
   \begin{aligned}
    1 , & \  \textrm{if $\star = A^{\BR}$}; \\
      -1, & \  \textrm{if $\star = A^{\BH}$}.
   \end{aligned}
   \right.$$
   The group $\G(V)$ is $\GL_n(\BR)$ when $\mathbf{j}^2 = 1$, and $\GL_{\frac{n}{2}}(\BH)$ when $\mathbf{j}^2 = -1$. Here $n = \mathrm{dim}(V)$.

    \textbf{The case when $\star \in \{A, \tilde{A}\}$}. In this case, a $\star$-structure consists of a finite-dimensional complex vector space $V$, and a non-degenerate Hermitian form $\langle -,- \rangle: V \times V \to \BC$ (which is linear on the first variable and conjugate linear on the second variable). The Hermitian form has signature $(p,q)$ so $\G(V)$ is $\U(p,q)$ with $p+q=n$.

    \textbf{The case when $\star = A^\BC$}. In this case, a $\star$-structure consists of an even-dimensional complex vector space
    $V$, with a decomposition $V = W \oplus W'$, and a conjugate linear isomorphism $j: W \to W'$. The group $\G(V)$ is $\GL_n(\BC)$.

   For a $\star$-space $V$, put
   $$
   \G_{\star}(V) := \left\{
    \begin{aligned}
         &\G(V), & \textrm{if $\star \in \{A^{\BR},A^{\BH},A^{\BC},A\}$};\\
         &\textrm{the $\mathrm{det}^{\frac{1}{2}}$-double cover of} \ \G(V), & \textrm{if $\star \in \{\tilde{A}\}$}.
    \end{aligned}
    \right.$$

    Now we assume that $G$ is identified with $\G_{\star}(V)$. We call $V$ the standard representation of $G$. Recall that $n$ is the rank of $\fg$ in all cases. 
   In the case when $\star \in \{A^\BR,A^\BH,A\}$, we fix a flag
\begin{equation}\label{flag}
   \{0\} = V_0 \subseteq V_1 \subseteq \cdots \subseteq V_n
\end{equation}
in $V$ such that $\mathrm{dim}(V_i) = i$ for all $i = 1,2,\cdots,n$.

When $\star = A^\BC$, we fix two flags
\begin{align}
    \{0\} = W_0 \subseteq W_1 \subseteq \cdots \subseteq W_n = W \\
    \{0\} = W'_0 \subseteq W_1' \subseteq \cdots \subseteq W_n' = W'
\end{align}
in $W$ and $W'$ such that $\dim W_i = \dim W_i' = i$ for $i = 0, 1, \cdots ,n$.

The stabilizer of the flag in $\fg$ is a Borel subalgebra of $\fg$. Using this Borel subalgebra, we get natural identifications
$$ {^{a}\fh} = \left\{
   \begin{aligned}
       &\prod^{n}_{i=1}\fgl(V_i/V_{i-1}) = \BC^n, & \textrm{if $\star \in \{A^\BR,A^\BH,A\}$};\\
       &\prod^{n}_{i=1}\fgl(W_i/W_{i-1}) \times \prod^{n}_{i=1}\fgl(W_{i}'/W_{i-1}') = \BC^n \times \BC^n, & \textrm{if $\star = A^\BC$}.
   \end{aligned}
   \right.$$
This identification is independent of the choice of flag (\ref{flag}). As in the Section \ref{2.1}, we have natural identifications:
$$
    Q_{\iota} = \left\{
    \begin{aligned}
        & \BZ^n \subseteq \BC^n = (\BC^n)^* = {^{a}\fh^*}, \ &\textrm{if $\star \in \{A^\BR, A^\BH,A\}$};\\
        & \BZ^n \times \BZ^n \subseteq \BC^n \times \BC^n = (\BC^n)^*\times(\BC^n)^* = {^{a}\fh^*}, \ &\textrm{if $\star = A^\BC$},
   \end{aligned}
   \right.
$$
and the positive roots are 
$$ \Delta^+ = \left\{
    \begin{aligned}
        &\set{e_i - e_j}{1 \leq i < j \leq n}, \ &\textrm{if $\star \in \{A^\BR,A^\BH,A\}$};\\
        &\set{e_i - e_j, \ e_i' - e_j'}{1\leq i<j\leq n}, \ & \textrm{if $\star = A^\BC$}.
    \end{aligned}
   \right.
$$
Where $e_1,e_2,\cdots,e_n$ and $e_1', e_2', \cdots, e_n'$ are both the standard basis of $\BC^n$, they also correspond to all the weights of the standard representation.
The abstract Weyl group 
$$
   W = \left\{
   \begin{aligned}
       &\mathrm{S}_n, & \textrm{if $\star \in \{A^\BR,A^\BH,A\}$};\\
       &\mathrm{S}_n \times \mathrm{S}_n, & \textrm{if $\star = A^\BC$}.
   \end{aligned}
   \right.
$$
Where $\mathrm{S}_n \subseteq \GL_n(\BZ)$ is the group of permutation matrices.

\subsection{Proof of Theorem \ref{R}, for $G = \GL_n(\BR)$}\label{3.2}
\subsubsection{Integral case}
If $\nu$ is an integral character in ${^{a}\fh}^*$, i.e. 
$$\nu = (\underbrace{\lambda_1, \cdots, \lambda_1}_{d_1}, \underbrace{\lambda_2, \cdots, \lambda_2}_{d_2}, \cdots, \underbrace{\lambda_k, \cdots, \lambda_k}_{d_k} ) \in {^{a}\fh}^* = \BC^n,$$ 
where $\lambda_i - \lambda_j \in \BZ$. In this case $W_{\Lambda} = W(\Lambda) = W$, $W_\nu = \mathrm{S}_{d_1} \times \cdots \times \mathrm{S}_{d_k}$. 

The set of conjugacy classes of the Cartan subgroup is parameterized by non-negative integers in $\{0,1,\cdots,\lfloor \frac{n}{2}\rfloor\}$, fix a parameter $r$, a representative element of this conjugacy class is given by a Cartan subgroup $H_r$, which is isomorphic to:
$$(\BC^{\times})^r \times (\BR^{\times})^{n-2r}$$
with the corresponding weights $f_1,f_2,\cdots ,f_n$ in the standard representation, where
\begin{align}
    &f_{2i-1}(z_1,\cdots,z_r,t_{2r+1},\cdots,t_{n}) = z_i, \ & \textrm{if $i \leq r$};\\
    &f_{2i}(z_1,\cdots,z_r,t_{2r+1},\cdots,t_{n}) = \bar{z_i}, \ & \textrm{if $i \leq r$};\\
    &f_{i}(z_1,\cdots,z_r,t_{2r+1},\cdots,t_{n}) = t_i, \ & \textrm{if $i \geq 2r$}.
\end{align}
The set of roots is given by
$$\set{\pm (f_i - f_j)}{1 \leq i < j \leq n},$$
with the subset of real roots
$$\set{\pm(f_i - f_j)}{2r \leq i < j \leq n},$$
and the subset of imaginary roots
$$\set{\pm(f_{2i-1} - f_{2i})}{1 \leq i \leq r}.$$
All imaginary roots are compact imaginary.

Denote the complexified Lie algebra of $H_r$ by $\fh_r$. Let $\xi \in W({^{a}\fh}^*,\fh_r^*)$ be the unique element such that $\xi(e_i) = f_i$. By abuse of notation, we may also identify $f_i$ with their differentials, which are linear functionals on $\fh_r$.

The real Weyl group $W_{H_r}$ of $H_r$ is given by
$$W_{H_r} = (\mathrm{S}_r \ltimes  (W(A_1))^r) \times \mathrm{S}_{n-2r},$$
where 
\begin{itemize}
    \item $\mathrm{S}_r$ is generated by $s_{f_{2i-1}-f_{2i+1}}s_{f_{2i}-f_{2i+2}}$ for $1 \leq i \leq r-1$;
    \item $W(A_1)^r$ is generated by $s_{f_{2i-1}-f_{2i}}$ for $1 \leq i \leq r$ which is the Weyl group of imaginary root system;
    \item $\mathrm{S}_{n-2r}$ is generated by $s_{f_{i}-f_{i+1}}$ for $n-2r +1 \leq i \leq n-1$.
\end{itemize}

Then the quadratic character $\sgn_{\mathrm{im}}$ of $W_{H_r}$ is

\begin{itemize}
    \item sign character on $(W(A_1))^r$;
    \item trivial character on other factors.
\end{itemize}

Let $\CP_r$ denote the set of parameters in $\CP_{\Lambda}(G)$ which are represented by triples of the form $(H,\zeta,\Gamma) \in \mathscr{P}_{\Lambda}(G)$, where $H$ conjugate to $H_r$.

\begin{lem}\label{gp char}
    For any Cartan subgroup $H$ and $\zeta \in W({^{a}\fh}^*,\fh^*)$, $\zeta(\nu) + \delta(\zeta) \in {^{a}\fh^*}$ is always the differential of a continuous character on $H$. 
\end{lem}

\begin{proof}
    Without lose of generality, assume $H = H_r$ for some $r$, and $\zeta = \xi$. Then, $\delta(\xi) = -\frac{1}{2} \cdot (\underbrace{1,-1,1,-1,\cdots,1,-1}_{2r},0,\cdots,0)$, where we identify ${^{a}\fh^*}$ with $\fh^*$ through $\xi$. 
    
    $\xi(\nu) = (\nu_1,\cdots,\nu_n)$, with $\nu_i \in \BZ$. The difference between $2i$-th entry and $(2i-1)$-th entry of $\zeta(\nu) + \delta(\zeta)$ is always an integer for $1 \leq i \leq r$, and hence $\zeta(\nu) + \delta(\zeta)$ comes from the differential of a group character.

\end{proof}

Under the cross action, the set of $W$-orbits in $\CP_r$ is parameterized by non-negative integers in $\{0,1,\cdots,n-2r\}$. Each $i \in \{0,1,\cdots,n-2r\}$ corresponds to a $W$-orbit of $\CP_r$ represented by $\gamma_{r,i} = G \cdot \upgamma_{r,i} = G \cdot (H_r,\xi,\Gamma^i) \in \CP_{\Lambda}(G)$, where $\Gamma_{\nu}^i|_{\{\pm1\}^{n-2r}} = \underbrace{1 \otimes \cdots \otimes 1}_{i} \otimes \underbrace{\sgn \otimes \cdots \otimes \sgn }_{n-2r-i}$. ($\{\pm1\}^{n-2r} \subseteq (\BR^\times)^{n-2r}$ is the component subgroup.)

We can identify $W$ with $W_{\fh_r}$ through the isomorphism $\xi$, then the centralizer $W_{{\gamma_{r,i}}}$ of ${\gamma_{r,i}} \in \CP_{\Lambda}(G)$ in $W$ is identified with the subgroup $\xi \circ W_{{\gamma_{r,i}}}\circ \xi = (\mathrm{S}_{r} \ltimes (W(A_1))^r) \times (\mathrm{S}_{i} \times \mathrm{S}_{n-2r-i}) \subseteq (\mathrm{S}_{r} \ltimes (W(A_1))^r) \times \mathrm{S}_{n-2r} = W_{H_r} \subseteq W_{\fh}$.

Now, using Theorem \ref{Coh}, we can give the coherent continuation representation as
\begin{equation}
    \mathrm{Coh}_{\Lambda}(\CK(G)) = \bigoplus_{2r + i \leq n} \Ind _{\mathrm{W}_{r} \times \mathrm{S}_{i} \times \mathrm{S}_{n-2r-i}}^{\mathrm{S}_{n}} \epsilon \otimes 1 \otimes 1.
\end{equation}

For a given nilpotent orbit $\CO = \CO_{\iota}$, $\Irr(W;\CO) = \{\phi(\iota)\}$ is a singleton, so 
\begin{equation}
    \sharp(\Irr_{\nu}(G;\CO)) = \left[1_{W_v}: \phi(\iota)\right]\cdot \left[\phi(\iota): \bigoplus_{2r + i \leq n} \Ind _{\mathrm{W}_{r} \times \mathrm{S}_{i} \times \mathrm{S}_{n-2r-i}}^{\mathrm{S}_{n}} \epsilon \otimes 1 \otimes 1\right].
\end{equation}
By Frobenius reciprocity, 
\begin{equation}
\left[1_{W_\nu}: \phi(\iota)\right] = \left[\phi(\iota): \Ind_{W_\nu}^{W}1\right] = \left[\phi(\iota): \Ind_{\mathrm{S}_{d_1} \times \cdots \times \mathrm{S}_{d_k}}^{\mathrm{S}_n}1\right] = \sharp\left(\A_{[d_1,\cdots,d_k]}(\iota)\right),
\end{equation}
the last equality follows from Pieri's rule \ref{Pieri}.

At the same time, by using Proposition \ref{Pieri}, \ref{branch}, and induction by step, we obtain
\begin{equation}
    \left[\phi(\iota): \bigoplus_{2r + i \leq n} \Ind _{\mathrm{W}_{r} \times \mathrm{S}_{i} \times \mathrm{S}_{n-2r-i}}^{\mathrm{S}_{n}} \epsilon \otimes 1 \otimes 1\right] = \sharp\left(\mathrm{P}_{A^\BR}(\iota)\right).
\end{equation}
Hence, we have proved the first part of Theorem \ref{R}.

\subsubsection{General case}
For the case of a possibly non-integral infinitesimal character, i.e. 
$$\nu = (\blam_1, \cdots, \blam_r) \in {^{a}\fh},$$ 
where 
$$\blam_i = (\underbrace{\lambda_{i,1}, \cdots, \lambda_{i,1}}_{d_{i,1}},\cdots,\underbrace{\lambda_{i,k_i},\cdots,\lambda_{i,k_i}}_{d_{i,k_i}}) \in \BC^{e_i},$$  $d_{i,1} \geq \cdots \geq d_{i,k_i} \geq 1$ is a partition of $e_i$, and $\lambda_{i,p} - \lambda_{i,q} \in \BZ$, $\lambda_{i,1} - \lambda_{j,1} \notin \BZ$ for different $i,j$. The integral Weyl group is $W_{\Lambda} = W(\Lambda) = \mathrm{S}_{e_1} \times \cdots \times \mathrm{S}_{e_r} \subseteq \mathrm{S}_{n} = W$. 

Fix a decomposition $V = V_1 \oplus V_2 \oplus \cdots \oplus V_r$ with $\dim V_i = e_i$ and each $V_i$ is stable under $\mathbf{j}$, this induces a subgroup $\G_{A^{\BR}}(V_1) \times \cdots \times \G_{A^{\BR}}(V_r)$ of $\G_{A^{\BR}}(V)$. Let $^{a}\fh_i$ be the abstract Cartan subalgebra of $\G_{A^{\BR}}(V_i)$, we have an identification
$${^{a}\fh_1^*} \times {^{a}\fh_2^*} \times \cdots \times {^{a}\fh_r^*} = \BC^{e_1} \times \BC^{e_2} \times \cdots \times \BC^{e_r} = \BC^n = {^{a}\fh^*},$$
under this identification $\blam_i \in {^{a}\fh_i^*}$, let  $Q_i \subseteq {^{a}\fh_i^*}$ be the analytical weight lattices of $\G_{A^{\BR}}(V_i)$ and $\Lambda_i := \lambda_i + Q_i \subseteq {^{a}\fh_i^*}$. There is also an identification of integral Weyl groups
$$W(\Lambda_1) \times W(\Lambda_2) \times \cdots \times W(\Lambda_r) = \mathrm{S}_{e_1} \times \cdots \times \mathrm{S}_{e_r} = W(\Lambda) = W_{\Lambda}.$$

For each $i$, define $\mathscr{P}_{\Lambda_i}(\G_{A^{\BR}}(V_i))$ to be the set of triples $(H_i, \xi_i, \Gamma_i)$, such that $H_i$ is a Cartan subgroup of $\G_{A^{\BR}}(V_i)$, $\xi_i \in W({^{a}\fh_i^*},\fh_i^*)$ where $\fh_i$ is the complexified Lie algebra of $H_i$, and \defmap{\Gamma_i}{\Lambda_i}{\Irr(H_i)}{\nu_i}{\Gamma_{\nu_i}}
which satisfies the same condition as Definition \ref{regular character}. Similarly, define $\CP_{\Lambda_i}(\G_{A^{\BR}}(V_i))$ to be the set of $\G_{A^{\BR}}(V_i)$-conjugacy classes under the standard adjoint action of $\G_{A^{\BR}}(V_i)$.

Now, we can define a map 
\begin{equation}\label{3.12}
    \begin{aligned}
        \varphi :\quad  & \CP_{\Lambda_i}(\G_{A^{\BR}}(V_1)) \times \cdots \times \CP_{\Lambda_r}(\G_{A^{\BR}}(V_r)) &\longrightarrow &\quad \CP_{\Lambda}(G) \\
        \quad  &  ((H_1,\xi_1,\Gamma_1),\cdots,(H_r,\xi_r,\Gamma_r))   &\longmapsto  &\quad (H,\xi,\Gamma)
    \end{aligned}
 \end{equation}

where $H = H_1 \times \cdots \times H_r \subseteq G$, and $\xi$ is the composition of
\[
    ^{a}\fh^*={^{a}\fh}_1^*\times \cdots \times  {^{a}\fh}_{r}^* \xrightarrow{\xi_1 \times \cdots \times \xi_r} \fh_1^* \times \cdots \times \fh_r^* = \fh^*,
\]
and $\Gamma: \Lambda \to \Irr(H)$ is the map
$$\nu = (\nu_1,\cdots ,\nu_r) \mapsto \Gamma_{1,\nu_1} \otimes \cdots \otimes \Gamma_{r,\nu_r}.$$

Similar to \cite[Proposition 6.13]{BMSZ}, we have the following proposition.

\begin{prop}\label{varphi bij}
    The map $\varphi$ defined above is bijective.
\end{prop}

\begin{proof}
    Note that for every triple $(H,\xi,\Gamma) \in \mathscr{P}_{\Lambda}(G)$, $H$ can be embedding into the Levi subgroup $\G_{A^{\BR}}(V_1) \times \cdots \times \G_{A^{\BR}}(V_r)$ after conjugation. The rest follows from direct verification.
\end{proof}

The bijection of parameters above induces a linear map:
\defmap{\varphi}{\mathrm{Coh}_{\Lambda_1}(\CK(\G_{A^{\BR}}(V_1))) \otimes \cdots \otimes \mathrm{Coh}_{\Lambda_r}(\CK(\G_{A^{\BR}}(V_r)))}{\mathrm{Coh}_{\Lambda}(\CK(G))}{\Psi_{\gamma_1} \otimes \cdots \otimes \Psi_{\gamma_n}}{\Psi_{\varphi(\gamma_1,\cdots,\gamma_r)}}

\begin{prop}\label{Coh cong}
    The linear map $\varphi$ defined above is an isomorphism of representations of $W_{\Lambda}$.
\end{prop}

\begin{proof}
    It is straightforward to verify that the bijection (\ref{3.12}) is also equivariant with respect to the cross action of $W_{\Lambda}$ and the operation known as the Cayley transformation, as defined in \cite[Chapter 8]{Vog81}. The proposition then follows from the explicit description of the coherent continuation representation with respect to the standard basis, as given in \cite[Definition 14.4]{Vog82}.
\end{proof}

Now, for an irreducible representation $\sigma = \phi(\iota_1) \boxtimes \cdots \boxtimes \phi(\iota_r)$ of $W_{\Lambda} = \mathrm{S}_{e_1} \times \cdots \times \mathrm{S}_{e_r}$, we have 
\begin{align*}
[\sigma:&\mathrm{Coh}_{\Lambda}(\CK(G))] \\
 & = [\phi(\iota_1) \boxtimes \cdots \boxtimes \phi(\iota_r) : \mathrm{Coh}_{\Lambda_1}(\CK(\G_{A^{\BR}}(V_1))) \boxtimes \cdots \boxtimes \mathrm{Coh}_{\Lambda_r}(\CK(\G_{A^{\BR}}(V_r)))]\\
 & = [\phi(\iota_1): \mathrm{Coh}_{\Lambda_1}(\CK(\G_{A^{\BR}}(V_1)))]\cdots [\phi(\iota_r):\mathrm{Coh}_{\Lambda_r}(\CK(\G_{A^{\BR}}(V_r)))]\\
 & = \sharp(\mathrm{P}_{A^\BR}(\iota_1)) \cdots \sharp(\mathrm{P}_{A^\BR}(\iota_r)),
\end{align*}
and use Frobenius reciprocity
\begin{align*}
    [1_{W_\nu} : \sigma ] & = [\sigma: \Ind_{W_{\nu}}^{W}1]\\
    & = [\phi(\iota_1) \boxtimes \cdots \boxtimes \phi(\iota_r): \Ind_{W_{\blam_1}}^{W_1}1 \boxtimes \cdots \boxtimes \Ind_{W_{\blam_r}}^{W_r}1 ]\\
    & = [\phi(\iota_1):\Ind_{W_{\blam_1}}^{W_1}1] \cdots [\phi(\iota_r):\Ind_{W_{\blam_r}}^{W_r}1]\\
    & = \sharp\left(\A_{[d_{1,1},\cdots,d_{1,r_1}]}(\iota_1)\right)\cdots \sharp\left(\A_{[d_{k,1},\cdots,d_{r,k_r}]}(\iota_r)\right),
\end{align*}
where $W_i$ is the corresponding abstract Weyl groups on $^{a}\fh_i$.

For a fixed nilpotent orbit $\CO = \CO_{\iota} \in \bar{\Nil}(\fg^*)$, 
\begin{align*}
    \Irr_{\CO_{\iota}}(W_{\Lambda}) & = \Irr_{\CO_{\iota}}(\mathrm{S}_{e_1} \times \cdots \times \mathrm{S}_{e_r})\\
    & = \set{\sigma \in \Irr(\mathrm{S}_{e_1} \times \cdots \times \mathrm{S}_{e_r})}{j_{\mathrm{S}_{e_1} \times \cdots \times \mathrm{S}_{e_r}}^{\mathrm{S}_n}\sigma = \phi(\iota)}\\
    & = \set{(\iota_1,\cdots,\iota_r) \in \mathrm{YD}_{e_1} \times \cdots \times \mathrm{YD}_{e_r}}{\iota_1 \mathop{\sqcup}\limits^r \iota_2 \cdots \mathop{\sqcup}\limits^r  \iota_r = \iota(\CO)},
\end{align*}
where the third equality follows from Proposition \ref{2.9}.

Combining the above three identities, we obtain the second part of Theorem \ref{R}.

\subsection{Proof of Theorem \ref{H}, for $G = \GL_{\frac{n}{2}}(\BH)$}
\subsubsection{Integral case}
If $\nu$ is an integral character in $^{a}\fh^*$, i.e. 
$$\nu =  (\underbrace{\lambda_1, \cdots, \lambda_1}_{d_1}, \underbrace{\lambda_2, \cdots, \lambda_2}_{d_2}, \cdots, \underbrace{\lambda_k, \cdots, \lambda_k}_{d_k} ) \in {^{a}\fh^*} = \BC^n,$$ 
where $d_1 \geq d_2 \geq \cdots \geq d_k \geq 1$ is a partition of $n$ and $\lambda_i - \lambda_j \in \BZ$. 

There is only one conjugacy class of Cartan subgroups, a representative is given by $H$, which is isomorphic to
$$(\BC^\times)^{\frac{n}{2}}$$
with the corresponding weights $f_1, f_2, \cdots ,f_n$ in the standard representation, where
\begin{align}
    &f_{2i-1}(z_1,\cdots,z_{\frac{n}{2}}) = z_i,\\
    &f_{2i}(z_1,\cdots,z_{\frac{n}{2}}) = \bar{z_i}.
\end{align}
The set of roots is given by
$$\set{\pm(f_i - f_j)}{1\leq i < j \leq n},$$
with no real roots, the subset of imaginary roots is
$$\set{\pm(f_{2i-1}-f_{2i})}{1 \leq i \leq \frac{n}{2}}.$$
All imaginary roots are compact imaginary.

Denote the complexified Lie algebra of $H$ by $\fh$, let $\xi \in W({^{a}\fh^*},\fh^*)$ be the unique element such that $\xi(e_i) = f_i$.

The real Weyl group $W_{H}$ is given by
$$W_{H} = \mathrm{S}_{\frac{n}{2}} \ltimes W(A_1)^{\frac{n}{2}},$$
where 
\begin{itemize}
    \item $\mathrm{S}_{\frac{n}{2}}$ is generated by $s_{f_{2i-1}-f_{2i+1}}s_{f_{2i}-f_{2i+2}}$;
    \item $W(A_1)^r$ is generated by $s_{f_{2i-1}-f_{2i}}$.
\end{itemize}

The quadratic character $\sgn_{\mathrm{im}}$ of $W_{H}$ is 
\begin{itemize}
    \item sign character on $W(A_1)^r$;
    \item trivial character on $\mathrm{S}_{\frac{n}{2}}$.
\end{itemize}

\begin{lem}\label{3.4}
For any Cartan subgroup $H'$ and any $\zeta \in W({^{a}\fh^*,\fh'^*})$, $\zeta(\nu) + \delta(\zeta)$ is always the differential of a continuous group character on $H'$.
\end{lem}

\begin{proof}
    Without loss of generality, we may assume $H' = H$ and $\zeta = \xi$. Then, $\xi(\nu) + \delta(\xi) = (\nu_1,\cdots,\nu_n) - \frac{1}{2}(1,-1,1,-1,\cdots,1,-1)$, where we identify ${^{a}\fh^*}$ with $\fh'^*$ through $\xi$, and $\nu_i \in \BZ$. It always comes from a group character.
\end{proof}

Since the abstract Weyl group $W$ acts transitively on the set $W({^{a}\fh}^*,\fh^*)$, and $\Gamma$ in a character $\upgamma' = (H,\xi',\Gamma)$ is uniquely determined by $H$ and $\xi'$, the cross action of $W$ on $\CP_{\Lambda}(G)$ is transitive.  Choose a representative element $\upgamma = (H,\xi,\Gamma)$, where $\Gamma$ is the unique map determined by $H$ and $\xi$.

We can identify $W$ with $W_{\fh}$ through the isomorphism $\xi$, then the centralizer $W_{{\gamma}}$ of $\gamma = G \cdot \upgamma$ in $W$ can be identified with $W_{H} \subseteq W_{\fh}$.

Now, using Theorem \ref{Coh}, we can give the coherent continuation representation as 
\begin{equation}
    \mathrm{Coh}_{\Lambda}(\CK(G)) = \Ind_{\mathrm{W}_{\frac{n}{2}}}^{\mathrm{S}_n} \epsilon.
\end{equation}

So for a given nilpotent orbit $\CO = \CO_{\iota}$, $\Irr(W;\CO) = \{\phi(\iota)\}$ is a singleton, we can use Pieri's rule \ref{Pieri}, and branching formula \ref{branch} again
\begin{equation}
    \sharp(\Irr_{\nu}(G;\CO)) = [\phi(\iota):\Ind_{W_\nu}^{W}1]\cdot[\phi(\iota): \Ind_{\mathrm{W}_{\frac{n}{2}}}^{\mathrm{S}_n}\epsilon] = \sharp(\A_{[d_1,\cdots,d_k]}(\iota))\cdot \sharp(\mathrm{P}_{A^\BH}(\iota)).
\end{equation}

Hence, we have proved the first part of Theorem \ref{H}.

\subsubsection{General case}
For the case of a possibly non-integral infinitesimal character, i.e. 
$$\nu = (\blam_1, \cdots, \blam_r) \in {^{a}\fh}^*,$$ 
where 
$$\blam_i = (\underbrace{\lambda_{i,1}, \cdots, \lambda_{i,1}}_{d_{i,1}},\cdots,\underbrace{\lambda_{i,k_i},\cdots,\lambda_{i,k_i}}_{d_{i,k_i}}) \in \BC^{e_i},$$ $d_{i,1} \geq \cdots \geq d_{i,k_i} \geq 1$ is a partition of $e_i$ and $\lambda_{i,p} - \lambda_{i,q} \in \BZ$, $\lambda_{i,1} - \lambda_{j,1} \notin \BZ$ for different $i,j$. The integral Weyl group is $W_{\Lambda} = W(\Lambda) = \mathrm{S}_{e_1} \times \cdots \times \mathrm{S}_{e_r} \subseteq \mathrm{S}_n$. 

If there exists any odd $e_i$, then $\xi(\nu) + \delta(\xi)$ cannot be the differential of a continuous character of $H$, as the coordinates with integral differences in a group character should appear in pairs. Consequently, there do not exist corresponding irreducible Casselman--Wallach representations. 

Now, we assume that all the $e_i$'s are even, fix a decomposition $V = V_1 \oplus \cdots \oplus V_r$ with $\dim V_{i} = e_i$ and each $V_i$ is stable under $\mathbf{j}$, it induces a subgroup $\G_{A^{\BH}}(V_1) \times \cdots \times \G_{A^{\BH}}(V_r)$ of $\G_{A^{\BH}}(V)$. Let $^{a}\fh_i$ be the abstract Cartan subalgebra of $\G_{A^{\BH}}(V_i)$, we have an identification
$${^{a}\fh_1^*} \times {^{a}\fh_2^*} \times \cdots \times {^{a}\fh_r^*} = \BC^{e_1} \times \BC^{e_2} \times \cdots \times \BC^{e_r} = \BC^n = {^{a}\fh^*},$$
under this identification $\blam_i \in {^{a}\fh_i^*}$, let  $Q_i \subseteq {^{a}\fh_i^*}$ be the analytical weight lattices of $\G_{A^{\BH}}(V_i)$ and $\Lambda_i := \lambda_i + Q_i \subseteq {^{a}\fh_i^*}$. There is also an identification of integral Weyl groups
$$W(\Lambda_1) \times W(\Lambda_2) \times \cdots \times W(\Lambda_r) = \mathrm{S}_{e_1} \times \cdots \times \mathrm{S}_{e_r} = W(\Lambda) = W_{\Lambda}.$$

As in Section \ref{3.2}, for each $i = 1,2,\cdots,r$, we define $\CP_{\Lambda_i}(\G_{A^{\BH}}(V_i))$ accordingly, and we can similarly define the following map
\defmap{\varphi}{\CP_{\Lambda_1}(\G_{A^{\BH}}(V_i))\times \cdots \times \CP_{\Lambda_r}(\G_{A^{\BH}}(V_r))}{\CP_{\Lambda}(G)}{((H_1,\xi_1,\Gamma_1), \cdots, (H_r,\xi_r,\Gamma_r))}{(H,\xi,\Gamma)}
using arguments analogous to those in Proposition \ref{varphi bij} and \ref{Coh cong}, we obtain the following results.
\begin{prop}
    The map $\varphi$ defined above is bijective.
\end{prop}

\begin{prop}
    There is a natural isomorphism of $W(\Lambda)$-representations
    \begin{equation}
        \varphi : \quad \mathrm{Coh}_{\Lambda_1}(\CK(\G_{A^{\BH}}(V_1))) \boxtimes \cdots \boxtimes \mathrm{Coh}_{\Lambda_r}(\CK(\G_{A^{\BH}}(V_r)))  \longrightarrow  \mathrm{Coh}_{\Lambda}(\CK(G)).
    \end{equation}
\end{prop}

For a nilpotent orbit $\CO = \CO_{\iota}$, we obtain

\begin{align}
    \sharp(\Irr_{\nu}(G;\CO)) & = \sum_{\sigma \in \Irr(W_\Lambda;\CO)}[1_{W_\nu}:\sigma]\cdot[\sigma:\mathrm{Coh}_{\Lambda}(\CK(G))]\\
    & = \sum_{\substack{(\iota_1,\cdots,\iota_r) \in \mathrm{YD}_{e_1} \times \cdots \times \mathrm{YD}_{e_r} \\ \iota_1 \mathop{\sqcup}\limits^r \cdots  \mathop{\sqcup}\limits^r \iota_r = \iota(\CO)}}\prod_{i=1}^r [\phi(\iota_i): \Ind_{W_{\blam_i}}^{W_i}1]\cdot \prod_{i=1}^r[\phi(\iota_i):\mathrm{Coh}_{\Lambda_i}(\CK(\G_{A^{\BH}}(V_i)))]\\
    & = \sum_{\substack{(\iota_1,\cdots,\iota_r) \in \mathrm{YD}_{e_1} \times \cdots \times \mathrm{YD}_{e_r} \\ \iota_1 \mathop{\sqcup}\limits^r \cdots  \mathop{\sqcup}\limits^r \iota_r = \iota(\CO)}} \prod_{i=1}^{r}\sharp\left(\mathrm{A}_{[d_{i,1}, \cdots,d_{i,k_i}]}(\iota_i)\right) \cdot \prod_{i=1}^r \sharp\left(\mathrm{P}_{A^\BH}(\iota_i)\right)\\
    & = \sum_{\substack{(\iota_1,\cdots,\iota_r) \in \mathrm{YD}_{e_1} \times \cdots \times \mathrm{YD}_{e_r} \\ \iota_1 \mathop{\sqcup}\limits^r \cdots  \mathop{\sqcup}\limits^r \iota_r = \iota(\CO)}} \prod_{i=1}^r \sharp\left(\Irr_{\blam_i}\left(\GL_{\frac{e_i}{2}}(\BH);\CO_{\iota_i}\right)\right).
\end{align}

Hence, the second part of Theorem \ref{H} follows.

\subsection{Proof of Theorem \ref{C}, for $G = \GL_n(\BC)$}

\subsubsection{Integral case}
If $\nu$ is an integral character in $^{a}\fh^*$, i.e. 
$$\nu =  (\underbrace{\lambda_1, \cdots, \lambda_1}_{d_1}, \cdots, \underbrace{\lambda_k, \cdots, \lambda_k}_{d_k}, \underbrace{\lambda_1', \cdots, \lambda_1'}_{d_1'}, \cdots, \underbrace{\lambda_k', \cdots, \lambda_k'}_{d_{l}'} ) \in {^{a}\fh^*} = \BC^n \times \BC^n,$$ 
where $[d_1,d_2,\cdots,d_k]$ and $[d_1',d_2',\cdots,d_l']$ are partitions of $n$, and $\lambda_i - \lambda_j, \ \lambda_i' - \lambda_j' \in \BZ$, then $W_{\Lambda} =W(\Lambda) = W = \mathrm{S}_n \times \mathrm{S}_n$, and $W_\nu = (\mathrm{S}_{d_1} \times \cdots \times \mathrm{S}_{d_k}) \times (\mathrm{S}_{d_1'} \times \cdots \times \mathrm{S}_{d_l'})$. 

There is only one conjugacy class of Cartan subgroups, a representative is given by $H$, which is isomorphic to
$$(\BC^\times)^n$$
with the corresponding weights $f_1,f_2,\cdots,f_n$ and $f_1',f_2',\cdots,f_n'$ in the standard representation, where
\begin{align}
    & f_i(z_1,z_2,\cdots,z_n) = z_i\\
    & f_i'(z_1,z_2,\cdots,z_n) = \bar{z_i}.
\end{align}
Then the set of roots is 
$$\set{\pm(f_i - f_j), \ \pm(f_i'-f_j')}{1\leq i< j \leq n}.$$
All of them are complex roots.

Denote the complexified Lie algebra of $H$ by $\fh$, fix a $\xi \in W({^{a}\fh^*},\fh^*)$ such that $\xi(e_i) = f_i$ and $\xi(e_i') = f_i'$.

The real Weyl group of $H$ is given by
$$W_H = \mathrm{S}_n,$$
which is generated by $s_{f_i-f_{i+1}}s_{f_i'-f_{i+1}'}$. 

The quadratic character $\sgn_{\mathrm{im}}$ of $W_{H}$ is the trivial character.

\begin{lem}\label{3.7}
    For any Cartan subgroup $H'$ and any $\zeta \in W({^{a}\fh^*,\fh'^*})$, $\zeta(\nu) + \delta(\zeta)$ is the differential of a continuous character of $H'$ if and only if $\lambda_1 - \lambda_1' \in \BZ$.
\end{lem}

\begin{proof}
    Without loss of generality, we may assume that $H' = H$ and $\zeta = \xi$. Then, $\xi(\nu) + \delta(\xi) = \xi(\nu) = (\lambda_1,\cdots,\lambda_n,\lambda_1',\cdots,\lambda_n')$, where we identify ${^{a}\fh^*}$ with $\fh'^*$ through $\xi$. It comes from a group character if and only if $\lambda_i - \lambda_i' \in \BZ$ for any $1 \leq i \leq n$, which is equivalent to $\lambda_1 - \lambda_1' \in \BZ$.
\end{proof}

If $\lambda_1 - \lambda_1' \notin \mathbb{Z}$, then there are no irreducible representations with this infinitesimal character.

Now, suppose that $\lambda_1 - \lambda_1' \in \BZ$. In this case, $W$ acts transitively on $\CP_{\Lambda}(G)$, we choose the representative element $\upgamma = (H,\xi, \Gamma)$, where $\Gamma$ is determined by $\xi$ since $H$ is connected.

We can identify $W$ with $W_\fh$ through the isomorphism $\xi$, then the centralizer $W_{{\gamma}}$ of ${\gamma} = G \cdot \upgamma \in \CP_{\Lambda}(G)$ in $W$ is identified with the real Weyl group $W_{H} = \mathrm{S}_n \subseteq \mathrm{S}_n \times \mathrm{S}_n = W_{\fh}$, the inclusion is diagonal embedding.

Applying Theorem~\ref{Coh}, we can give the coherent continuation representation as 
\begin{equation}
    \mathrm{Coh}_{\Lambda}(\CK(G)) = \Ind^{\mathrm{S}_n \times \mathrm{S}_n}_{\mathrm{S}_n}1.
\end{equation}

Given a nilpotent orbit $\CO = \CO_{\iota} \times \CO_{\iota'}$, $\Irr(W,\CO) = \{\phi(\iota) \boxtimes \phi(\iota')\}$ is a singleton, so
\begin{align}
    \sharp(\Irr_{\nu}(G;\CO)) & = [1_{W_\nu}:\phi(\iota) \boxtimes \phi(\iota')]\cdot[\phi(\iota)\boxtimes \phi(\iota'):\Ind^{\mathrm{S}_n \times \mathrm{S}_n}_{\mathrm{S}_n}1]\\
    & = [\phi(\iota)\boxtimes \phi(\iota'):\Ind_{W_\nu}^{W}1] \cdot [1_{\mathrm{S}_n}:\phi(\iota) \otimes \phi(\iota)]\\
    & = \sharp\left(\A_{[d_1,\cdots,d_k]}(\iota)\right) \cdot \sharp\left(\A_{[d_1',\cdots,d_l']}(\iota')\right) \cdot \delta_{\iota,\iota'},
\end{align}
the last equation is because all irreducible representations of symmetry groups are self-dual.

Hence, we have proved the first part of Theorem \ref{C}.

\subsubsection{General case}
For the case of a possibly non-integral infinitesimal character, i.e. 
$$\nu = (\blam_1, \cdots, \blam_r, \blam_1',\cdots, \blam_s') \in {^{a}\fh} = \BC^n \times \BC^n,$$ 
where 
\begin{align}
    &\blam_i = (\underbrace{\lambda_{i,1}, \cdots,\lambda_{i,1}}_{d_{i,1}},\cdots,\underbrace{\lambda_{i,k_i},\cdots,\lambda_{i,k_i}}_{d_{i,k_i}}) \in \BC^{e_i}\\
    &\blam_j' = (\underbrace{\lambda_{j,1},\cdots,\lambda_{j,1}}_{d_{j,1}'}, \cdots ,\underbrace{\lambda_{j,l_j},\cdots, \lambda_{j,l_j}}_{d_{j,l_j}'}) \in \BC^{e_j'},
\end{align}
 $[d_{i,1} , \cdots , d_{i,k_i}]$ is a partition of $e_i$, $[d_{j,1}', \cdots , d_{j,l_j}']$ is a partition of $e_j'$, and $\lambda_{i,p} - \lambda_{i,q} \in \BZ$, $\lambda_{i,1} - \lambda_{j,1}, \ \lambda_{i,1}' - \lambda_{j,1}' \notin \BZ$ for different $i,j$. The integral Weyl group is $W_\Lambda = W(\Lambda) = (\mathrm{S}_{e_1} \times \cdots \times \mathrm{S}_{e_r}) \times (\mathrm{S}_{e_1} \times \cdots \times \mathrm{S}_{e_r}) \subseteq \mathrm{S}_n \times \mathrm{S}_n = W$.

If there exist $\xi \in W({^{a}\fh^*},\fh^*)$ such that $\xi(\nu) + \delta(\xi)$ becomes the differential of a character of the group $H$, there must be a permutation of $\nu$ under $W$, such that $r = s$, $e_i = e_i'$ for all $i = 1,2,\cdots,r$, and $\lambda_{i,1} - \lambda_{i,1}' \in \BZ$, otherwise there will be no irreducible Casselman--Wallach representations with infinitesimal character $\nu$.

Now, we assume that $r = s$, $e_i = e_i'$ for each $i =1,2,\cdots,r$, and $\lambda_{i,1} - \lambda_{i,1}' \in \BZ$. Fix a decomposition $V = W \oplus W' = (W_1 \oplus W_1') \oplus \cdots \oplus (W_r \oplus W_r') = V_1 \oplus \cdots \oplus V_r$, where $\dim W_i = \dim W_i' = e_i$ and $j(W_i) = W_i'$, this decomposition corresponds to a subgroup $\G_{A^{\BC}}(V_1) \times \G_{A^{\BC}}(V_2) \times \cdots \times \G_{A^{\BC}}(V_r)$ of $G$. Let $^{a}\fh_i$ denote the abstract Cartan subalgebra of $\G(V_i)$, we have an identification
$${^{a}\fh_1^*} \times {^{a}\fh_2^*} \times \cdots \times {^{a}\fh_r^*} = \BC^{e_1} \times \BC^{e_1} \times \BC^{e_2} \times \BC^{e_2} \times \cdots \times \BC^{e_r} \times \BC^{e_r} = \BC^n \times \BC^n = {^{a}\fh^*},$$
under this identification $\blam_i \in {^{a}\fh_i^*}$, let  $Q_i \subseteq {^{a}\fh_i^*}$ be the analytical weight lattices of $\G_{A^{\BC}}(V_i)$ and $\Lambda_i := \lambda_i + Q_i \subseteq {^{a}\fh_i^*}$. There is also an identification of integral Weyl groups
$$W(\Lambda_1) \times W(\Lambda_2) \times \cdots \times W(\Lambda_r) = \mathrm{S}_{e_1} \times \mathrm{S}_{e_1} \times \cdots \times \mathrm{S}_{e_r} \times \mathrm{S}_{e_r} = W(\Lambda) = W_{\Lambda}.$$

As in Section \ref{3.2}, for each $i = 1,2,\cdots,r$, we define $\CP_{\Lambda_i}(\G_{A^{\BC}}(V_i))$ accordingly, and we can similarly define the following map
\defmap{\varphi}{\CP_{\Lambda_1}(\G_{A^{\BC}}(V_i))\times \cdots \times \CP_{\Lambda_r}(\G_{A^{\BC}}(V_r))}{\CP_{\Lambda}(G)}{((H_1,\xi_1,\Gamma_1),\cdots, (H_r,\xi_r,\Gamma_r))}{(H,\xi,\Gamma)}
using arguments analogous to those in Proposition \ref{varphi bij} and \ref{Coh cong}, we obtain the following results.
\begin{prop}
    The map $\varphi$ defined above is bijective.
\end{prop}

\begin{prop}
    There is a natural isomorphism of $W(\Lambda)$-representations
    \begin{equation}
        \varphi : \quad   \mathrm{Coh}_{\Lambda_1}(\CK(\G_{A^{\BC}}(V_1))) \boxtimes \cdots \boxtimes \mathrm{Coh}_{\Lambda_r}(\CK(\G_{A^{\BC}}(V_r)))  \longrightarrow \mathrm{Coh}_{\Lambda}(\CK(G)).
    \end{equation}
\end{prop}

For a nilpotent orbit $\CO = \CO_{\iota} \times \CO_{\iota'}$, we obtain

\begin{align}
    \sharp(\Irr_{\nu}(G;\CO)) & = \sum_{\sigma \in \Irr(W_{\Lambda};\CO)}[1_{W_\nu}:\sigma]\cdot[\sigma: \mathrm{Coh}_{\Lambda}(\CK(G))]\\
    & = \sum_{\substack{(\iota_1,\cdots,\iota_r), (\iota_1',\cdots,\iota_r') \in \mathrm{YD}_{e_1} \times \cdots \times \mathrm{YD}_{e_r}\\ \iota_1 \mathop{\sqcup}\limits^r \cdots \mathop{\sqcup}\limits^r \iota_r = \iota, \ \iota_1'  \mathop{\sqcup}\limits^r \cdots \mathop{\sqcup}\limits^r \iota_r' = \iota' }} \prod_{i=1}^{r} \sharp\left(\Irr_{\blam_i}(\GL_{e_i}(\BC);\CO_{\iota_i} \times \CO_{\iota_i'})\right).
\end{align}

Hence, the second part of Theorem \ref{C} follows.

\subsection{Proof of Theorem \ref{U}, for $G = \mathrm{U}(p,q)$}

\subsubsection{Integral case}
If $\nu$ is an integral character in ${^{a}\fh}^*$, i.e. 
$$\nu = (\underbrace{\lambda_1, \cdots, \lambda_1}_{d_1}, \underbrace{\lambda_2, \cdots, \lambda_2}_{d_2}, \cdots, \underbrace{\lambda_k, \cdots, \lambda_k}_{d_k} ) \in {^{a}\fh}^* = \BC^n,$$ 
where $\lambda_i - \lambda_j \in \BZ$. In this case $W_{\Lambda} = W(\Lambda) = W$, $W_\nu = \mathrm{S}_{d_1} \times \cdots \times \mathrm{S}_{d_k}$.

The set of conjugacy classes of Cartan subgroups is parameterized by non-negative integers in $\{0,1,\cdots, \mathrm{min}\{p,q\} \}$. Fix a parameter $s$, a representative element of the conjugacy class is given by a Cartan subgroup $H_s$, which is isomorphic to 
$$(\BC^{\times})^s \times (\BS^1)^{p-s} \times (\BS^1)^{q-s}$$
with the corresponding weights in the standard representation $f_1,\cdots,f_n$ where
\begin{align}
    & f_{i}(z_1,\cdots,z_s,a_1,\cdots,a_{p-s},b_1,\cdots,b_{q-s}) = z_i, & \textrm{if $i \leq s$};\\
    & f_{i}(z_1,\cdots,z_s,a_1,\cdots,a_{p-s},b_1,\cdots,b_{q-s}) = \bar{z_i}^{-1}, & \textrm{if $s+1 \leq i \leq 2s$};\\
    &f_i(z_1,\cdots,z_s,a_1,\cdots,a_{p-s},b_1,\cdots,b_{q-s}) = a_{i-2s}, & \textrm{if $2s+1 \leq i \leq s+p$};\\
    &f_i(z_1,\cdots,z_s,a_1,\cdots,a_{p-s},b_1,\cdots,b_{q-s}) = b_{i-s-p}, & \textrm{if $s+p+1 \leq i \leq n$}.
\end{align}
The set of roots is 
$$\set{\pm(f_i-f_j)}{1\leq i <j \leq n},$$
there are no real roots, and the subset of imaginary roots is
$$\set{\pm(f_i-f_j)}{2s+1 \leq i < j \leq n}.$$
Moreover, the subset of compact imaginary roots is
$$\set{\pm(f_i-f_j)}{2s+1 \leq i < j \leq s+p \  \textrm{or} \ s+p+1 \leq i<j \leq n }.$$
Denote the complexified Lie algebra of $H_s$ by $\fh_s$. 

The real Weyl group $W_{H_s}$ is given by
$$W_{H_s} = (\mathrm{S}_s \ltimes  (W(A_1))^s) \times \mathrm{S}_{p-s} \times \mathrm{S}_{q-s},$$

\begin{itemize}
    \item $\mathrm{S}_s$ is generated by $s_{f_{2i-1}-f_{2i+1}}s_{f_{2i}-f_{2i+2}}$ for $1 \leq i \leq s-1$;
    \item $W(A_1)^r$ is generated by $s_{f_{2i-1}-f_{2i}}$ for $1 \leq i \leq s$ which is the Weyl group of real root system;
    \item $\mathrm{S}_{p-s}$ (resp. $\mathrm{S}_{q-s}$) is generated by $s_{f_{i}-f_{i+1}}$ for $2r+1 \leq i \leq s+p-1$ (resp. $s+p+1 \leq i \leq n-1$).
\end{itemize}

The quadratic character $\sgn_{\mathrm{im}}$ of $W_{H_{s}}$ is
\begin{itemize}
    \item sign character on $\mathrm{S}_{p-s}$ and $\mathrm{S}_{q-S}$;
    \item trivial character on $\mathrm{S}_s \ltimes  (W(A_1))^s$.
\end{itemize}

Let $\CP_s$ denote the set of parameters in $\CP_{\Lambda}(G)$ which are represented by triples of the form $(H,\xi,\Gamma)$, where $H$ conjugate to $H_s$.

\begin{lem}
    \begin{itemize}
        \item If $\lambda_i \in \frac{n-1}{2} + \BZ$, for any Cartan subgroup $H$ and $\zeta \in W({^{a}\fh^*},\fh^*)$, $\zeta(\nu) + \delta(\zeta) \in {^{a}\fh^*}$ is always the differential of a continuous character of $H$;
        \item If $\lambda_i \in \frac{n}{2} + \BZ$, for a Cartan subgroup $H$, there exists a $\zeta \in W({^{a}\fh^*},\fh^*)$ such that $\zeta(\nu) + \delta(\zeta)$ is the differential of a continuous character of $H$ if and only if $p=q$ and $H$ conjugate to $H_p$;
        \item If $\lambda \notin \frac{1}{2}\BZ$, for any Cartan subgroup $H$ and $\zeta \in W({^{a}\fh^*},\fh^*)$, $\zeta(\nu) + \delta(\zeta) \in {^{a}\fh^*}$ is never the differential of a continuous character of $H$.
    \end{itemize}
\end{lem}

\begin{proof}
    The verification is similar to the proofs of Lemmas~\ref{gp char}, \ref{3.4}, and~\ref{3.7}, so we omit the details.
\end{proof}

Now, for a nilpotent orbit $\CO = \CO_{\iota}$.

If $\lambda_i \in \frac{n-1}{2} + \BZ$ for all $i = 1,2,\cdots,k$. The cross action of $W$ on $\CP_s$ is transitive, we choose the representative element $\upgamma = (H_s,\xi, \Gamma)$ for each $\CP_s$, where $\Gamma$ is determined by $\xi$ since $H_s$ is connected. We can identify $W$ with $W_{\fh_s}$ through the isomorphism $\xi$, then the centralizer $W_{{\gamma}}$ of ${\gamma} = G \cdot \upgamma \in \CP_s$ in $W$ is identified with $W_{H_s} \subseteq W_{\fh_s}$. Now, using Theorem \ref{Coh}, we can give the coherent continuation representation as
\begin{equation}
    \mathrm{Coh}_{\Lambda}(\CK(G)) = \bigoplus_{s = 0}^{\mathrm{min}(p,q)} \Ind^{\mathrm{S}_n}_{(\mathrm{S}_s \ltimes  (W(A_1))^s) \times \mathrm{S}_{p-s} \times \mathrm{S}_{q-s}}1 \boxtimes \mathrm{sgn} \boxtimes \mathrm{sgn} .
\end{equation}
Applying Pieri's rule~\ref{Pieri} and the branching formula~\ref{branch}, we obtain the number of irreducible representations as
\begin{align*}
    \sharp(\Irr_{\nu}&(G;\CO))\\ 
    &= \left[\phi(\iota): \Ind_{W_\nu}^{W}1\right] \cdot \left[\phi(\iota):\bigoplus_{s = 0}^{\mathrm{min}(p,q)} \Ind^{\mathrm{S}_n}_{(\mathrm{S}_s \ltimes  (W(A_1))^s) \times \mathrm{S}_{p-s} \times \mathrm{S}_{q-s}}1 \boxtimes \mathrm{sgn} \boxtimes \mathrm{sgn}\right]\\
    &= \sharp\left(\mathrm{P}_{A}^{p,q}(\iota(\CO))\right) \cdot \sharp\left(\mathrm{A}_{[d_1,\cdots,d_{k}]}(\iota(\CO))\right).
\end{align*}

If $\lambda_i \in \frac{n}{2} + \BZ$ for all $i = 1,2,\cdots,k$, by the lemma above, there exist irreducible representations of this infinitesimal character only if $p=q$ and $H$ is conjugate to $H_p$. In this case, the centralizer $W_{{\gamma}}$ of ${\gamma} = G \cdot \upgamma \in \CP_p$ in $W$ is identified with $W_{H_p} \subseteq W_{\fh_p}$. Now, using Theorem \ref{Coh}, we can give the coherent continuation representation as
$$\mathrm{Coh}_{\Lambda}(\CK(G)) = \Ind_{\mathrm{W_p}}^{\mathrm{S}_{2p}} 1.$$
Again, using Pieri's rule \ref{Pieri} and branching formula \ref{branch}, we obtain the number of irreducible representations as
\begin{align*}
    \sharp(\Irr_{\nu}&(G;\CO))\\ 
    &= [\phi(\iota): \Ind_{W_\nu}^{W}1] \cdot [\phi(\iota):\Ind_{\mathrm{W_p}}^{\mathrm{S}_{2p}} 1]\\
    &= \sharp(\mathrm{A}_{[d_1,\cdots,d_k]}(\iota(\CO))) \cdot \mathbf{1}_{\{\br_j(\iota(\CO))\in 2\BZ\ \text{for all }j\geq 1\}}.
\end{align*}

Finally, if $\lambda_i \notin \frac{1}{2}\BZ$ for some $i$, there will be no irreducible representation of this infinitesimal character.

\subsubsection{General case}
For the case of a possibly non-integral infinitesimal character $\nu$.
\begin{lem}
    The parameter set $\mathscr{P}_{\Lambda}(G)$ is nonempty only if the coordinates of $\nu$ can be permuted under the Weyl group $W$ such that it has the form
\begin{equation}
    \nu = (\blam,\blam', \blam_1, \blam_1', \cdots, \blam_r, \blam_r') \in {^{a}\fh}^* = \BC^n,
\end{equation}
where 
\begin{align}
    &\blam = (\lambda_{1}, \lambda_{2} \cdots , \lambda_{l} ) \in \left(\frac{n-1}{2}+ \BZ\right)^{l},\\
    &\blam' = (\lambda_{1}', \lambda_{2}' \cdots , \lambda_{l'}' ) \in \left(\frac{n}{2}+ \BZ\right)^{l'},\\
    &\blam_i = (\lambda_{i,1}, \lambda_{i,2} \cdots , \lambda_{i,l_i} ) \in \BC^{l_i},\\
    &\blam_i' = (\lambda_{i,1}', \lambda_{i,2}' \cdots , \lambda_{i,l_i}' ) \in \BC^{l_i},
\end{align}
and, for all $1 \le a,b \le l_i$, the following conditions hold:
\[
\lambda_{i,a} - \lambda_{i,b} \in \mathbb{Z}, \qquad
\lambda_{i,a}' - \lambda_{i,b}' \in \mathbb{Z}, \qquad
\lambda_{i,a} + \lambda_{i,b}' \in \mathbb{Z}, \qquad
\lambda_{i,a} \notin \tfrac{1}{2}\mathbb{Z},
\]
and $l'$ is even. The nonempty integral blocks are taken to be maximal modulo $\BZ$ among the coordinates of $\nu$, and $\lambda_{i,1}+\lambda_{j,1}\notin\BZ$ for all $1\le i<j\le r$.
\end{lem}

\begin{proof}
    It suffices to observe that the parameter set is nonempty only when $\nu+\rho$ is the differential of a character of the Cartan subgroup.  
    The stated conditions then follow from a direct verification.
\end{proof}

If the infinitesimal $\nu$ has the form above, the integral Weyl group is $W_{\Lambda} = W(\Lambda) = \mathrm{S}_{l} \times \mathrm{S}_{l'} \times \mathrm{S}_{l_1} \times \mathrm{S}_{l_1} \times \cdots \times \mathrm{S}_{l_r} \times \mathrm{S}_{l_r} \subseteq \mathrm{S}_n = W$.

Fix a decomposition  
\[
V = W \oplus W' \oplus W_1 \oplus W_1' \oplus \cdots \oplus W_r \oplus W_r',
\]  
where \(\dim W = l\), and the Hermitian form restricts to a nondegenerate form with signature \((p', q')\) on it, where  
\[
p' = p - \left(\frac{l'}{2} + l_1 + \cdots + l_r \right), \quad 
q' = q - \left(\frac{l'}{2} + l_1 + \cdots + l_r \right).
\]  
Similarly, \(\dim W' = l'\), and the Hermitian form restricts to a nondegenerate form with signature \((\frac{l'}{2}, \frac{l'}{2})\) on it.  
For each \( i \), we have \(\dim W_i = \dim W_i' = l_i\), where \( W_i \) and \( W_i' \) are isotropic, and \( W_i \oplus W_i' \) is nondegenerate with signature \((l_i, l_i)\).  

This decomposition corresponds to a group  
\[
\G_{A}(W) \times \G_{\star}(W') \times \GL(W_1) \times \cdots \times \GL(W_r),
\]  
here we set $\star = A$ if $p+q$ is even, and $\star = \tilde{A}$ if $p+q$ is odd.

Let \( {^{a}\fh_0} \) and \( {^{a}\fh_0'} \) denote the abstract Cartan subalgebras of \( \G_{A}(W) \) and \( \G_{\star}(W') \), respectively.  
Let \( {^{a}\fh_i} \) denote the abstract Cartan subalgebra of \( \GL(W_i) \).  
Then we have an identification
$${^{a}\fh_0^*} \times {^{a}\fh_0'^*} \times {^{a}\fh_1^*} \times {^{a}\fh_r^*} = \BC^{l} \times \BC^{l'} \times \BC^{l_1} \times \BC^{l_1} \times \cdots \times \BC^{l_r} \times \BC^{l_r} = \BC^{n} = {^{a}\fh^*},$$
under this identification, which sends $(\mathbf a_i,\mathbf b_i)$ in the $i$-th complex factor to $(\mathbf a_i,-\mathbf b_i)$ in the unitary coordinates, we have \( \blam \in {^{a}\fh_0^*} \), \( \blam' \in {^{a}\fh_0'^*} \), and \( (\blam_i,-\blam_i') \in {^{a}\fh_i^*} \), let \( Q_0 \), \( Q_0' \), and \( Q_i \) be the analytical weight lattices of \( \G_{A}(W) \), \( \G_{\star}(W') \), and \( \GL(W_i) \), respectively, and let \( \Lambda_0 = \blam + Q_0 \), \( \Lambda_0' = \blam' + Q_0' \), and \( \Lambda_i = (\blam_i,-\blam_i') + Q_i \). There is an identification of integral Weyl groups
$$W(\Lambda_0) \times W(\Lambda_0') \times W(\Lambda_1) \times \cdots \times W(\Lambda_r) = \mathrm{S}_{l} \times \mathrm{S}_{l'} \times \mathrm{S}_{l_1} \times \mathrm{S}_{l_1} \times \cdots \times \mathrm{S}_{l_r} \times \mathrm{S}_{l_r} = W(\Lambda).$$

We can define $\CP_{\Lambda_0}(\G_{A}(W))$, $\CP_{\Lambda_0'}(\G_{A}(W'))$, $\CP_{\Lambda_0'}'(\G_{\tilde{A}}(W'))$ , and \( \CP_{\Lambda_i}(\GL(W_i)) \) for each \( i = 1,2,\cdots,r \), where $\CP_{\Lambda_0'}'(\G_{\tilde{A}}(W'))$ denotes the set of genuine parameters in $\CP_{\Lambda_0'}(\G_{\tilde{A}}(W'))$.

If $p+q$ is even, we can define the following map in the same way as in Section \ref{3.2}.
\defmap{\varphi}{\CP_{\Lambda_0}(\G_{A}(W))\times \CP_{\Lambda_0'}(\G_{A}(W'))\times \CP_{\Lambda_1}(\GL(W_1))\times \cdots \times \CP_{\Lambda_r}(\GL(W_r))}{\CP_{\Lambda}(G)}{((H_0,\xi_0,\Gamma_0), (H_0',\xi_0',\Gamma_0'), (H_1,\xi_1,\Gamma_1), \cdots, (H_r,\xi_r,\Gamma_r))}{(H,\xi,\Gamma)}

If $p+q$ is odd, we can define the following map
\defmap{\varphi}{\CP_{\Lambda_0}(\G_{A}(W))\times \CP_{\Lambda_0'}'(\G_{\tilde{A}}(W'))\times \CP_{\Lambda_1}(\GL(W_1))\times \cdots \times \CP_{\Lambda_r}(\GL(W_r))}{\CP_{\Lambda}(G)}{((H_0,\xi_0,\Gamma_0), (H_0',\xi_0',\Gamma_0'), (H_1,\xi_1,\Gamma_1), \cdots, (H_r,\xi_r,\Gamma_r))}{(H,\xi,\Gamma)}
where $H = H_0 \times (H_0')_0 \times H_1 \times \cdots \times H_r$, $(H_0')_0$ is the identity component of $H_0'$. $\xi$ and $\Gamma$ are defined similarly as in the even case. 

Using arguments analogous to those in Proposition \ref{varphi bij} and \ref{Coh cong}, we obtain the following results.
\begin{prop}
    The map \(\varphi\) defined above is bijective. 
\end{prop}

\begin{prop}
    There is a natural isomorphism of \(W(\Lambda)\)-representations
    \begin{align*}
        \varphi : \quad  &\mathrm{Coh}_{\Lambda_0}(\CK(\G_{A}(W))) \boxtimes \mathrm{Coh}_{\Lambda_0'}(\CK(\G_{A}(W'))) \boxtimes \mathrm{Coh}_{\Lambda_1}(\CK(\GL(W_1))) \boxtimes \cdots \\ 
        &\boxtimes \mathrm{Coh}_{\Lambda_r}(\CK(\GL(W_r))) \longrightarrow \mathrm{Coh}_{\Lambda}(\CK(G)),
    \end{align*}
    if $p+q$ is even, and
    \begin{align*}
        \varphi : \quad  &\mathrm{Coh}_{\Lambda_0}(\CK(\G_{A}(W))) \boxtimes \mathrm{Coh}_{\Lambda_0'}(\CK^{\mathrm{gen}}(\G_{\tilde{A}}(W'))) \boxtimes \mathrm{Coh}_{\Lambda_1}(\CK(\GL(W_1))) \boxtimes \cdots \\ 
        &\boxtimes \mathrm{Coh}_{\Lambda_r}(\CK(\GL(W_r))) \longrightarrow \mathrm{Coh}_{\Lambda}(\CK(G)),
    \end{align*}
    if $p+q$ is odd.
\end{prop}  

For a nilpotent orbit $\CO = \CO_{\iota}$, if $p+q$ is even we obtain
\begin{align*}
    \sharp(\Irr_{\nu}&(G;\CO))\\ 
    &= \sum_{\sigma \in \Irr(W_{\Lambda};\CO)}[1_{W_\nu}:\sigma]\cdot[\sigma: \mathrm{Coh}_{\Lambda}(\CK(G))]\\
    &=  \sum_{\substack{(\iota, \iota', \iota_1,\iota_1' \cdots,\iota_r, \iota_r') \in \mathrm{YD}_{l} \times \mathrm{YD}_{l'} \times \mathrm{YD}_{l_1} \times \mathrm{YD}_{l_1} \times \cdots \times \mathrm{YD}_{l_r} \times \mathrm{YD}_{l_r}  \\ \iota \mathop{\sqcup}\limits^r \iota' \mathop{\sqcup}\limits^r \iota_1 \mathop{\sqcup}\limits^r \iota_1 \mathop{\sqcup}\limits^r \cdots  \mathop{\sqcup}\limits^r \iota_r \mathop{\sqcup}\limits^r \iota_r   = \iota(\CO)}} \sharp\left(\Irr_{\blam}\left(\U(p',q');\CO( \iota )\right)\right) \\
    & \cdot \sharp\left(\Irr_{\blam'}\left(\U\left(\frac{l'}{2},\frac{l'}{2}\right);\CO(\iota')\right)\right)\cdot \prod_{i=1}^{r} \sharp\left(\Irr_{(\blam_i,-\blam_i')}(\GL_{l_i}(\BC);\CO(\iota_i)\times \CO(\iota_i'))\right).
\end{align*}

If $p+q$ is odd, we obtain
\begin{align*}
    \sharp ( \Irr_{\nu}& ( G;\CO ))  \\ 
    &= \sum_{\sigma \in \Irr(W_{\Lambda};\CO)}[1_{W_\nu}:\sigma]\cdot[\sigma: \mathrm{Coh}_{\Lambda}(\CK(G))]\\
    &= \sum_{\substack{(\iota, \iota', \iota_1,\iota_1' \cdots,\iota_r, \iota_r') \in \mathrm{YD}_{l} \times \mathrm{YD}_{l'} \times \mathrm{YD}_{l_1} \times \mathrm{YD}_{l_1} \times \cdots \times \mathrm{YD}_{l_r} \times \mathrm{YD}_{l_r}  \\ \iota \mathop{\sqcup}\limits^r \iota' \mathop{\sqcup}\limits^r \iota_1 \mathop{\sqcup}\limits^r \iota_1 \mathop{\sqcup}\limits^r \cdots  \mathop{\sqcup}\limits^r \iota_r \mathop{\sqcup}\limits^r \iota_r   = \iota(\CO)}} \sharp(\Irr_{\blam}(\U(p',q');\CO(\iota)))\\
    & \cdot \sharp\left(\Irr_{\blam'}^{\mathrm{gen}}\left(\tilde{\U}\left(\frac{l'}{2},\frac{l'}{2}\right);\CO(\iota')\right)\right)\cdot \prod_{i=1}^{r} \sharp\left(\Irr_{(\blam_i,-\blam_i')}(\GL_{l_i}(\BC);\CO(\iota_i)\times \CO(\iota_i'))\right).
\end{align*}
Hence, the second part of Theorem \ref{U} follows.

\section*{Acknowledgements}
The author would like to thank Binyong Sun for introducing the problem and for his constant support and encouragement. The author is also grateful to Jiajun Ma for helpful discussions.

\end{document}